\documentclass[11pt, twoside]{article}
\pdfoutput=1

\usepackage[margin=1in]{geometry}

\usepackage{amsmath}
\usepackage{graphicx}  \usepackage{epstopdf}

\usepackage{dsfont}
\usepackage{subfig}

\usepackage{amsthm,amssymb}
\usepackage{cite}
\usepackage{hyperref}

\usepackage[nameinlink]{cleveref}

\numberwithin{equation}{section}

\newtheorem{theorem}{Theorem}[section]
\newtheorem{lemma}[theorem]{Lemma}

\theoremstyle{definition}
\newtheorem{definition}[theorem]{Definition}

\theoremstyle{remark}
\newtheorem{remark}[theorem]{Remark}

\usepackage{fancyhdr}
\title{Balanced Truncation Model Reduction of a Nonlinear Cable-Mass PDE System with Interior Damping}

\author{Belinda A.~Batten%
	\thanks{School of Mechanical, Industrial, \& Manufacturing Engineering, Oregon State University, Corvallis, OR 97331-6011, USA (bbatten@engr.orst.edu, hesam@servotechinc.com). B.~Batten was supported in part by the Department of Energy under Award Number DE-FG36-08GO18179.  H.~Shoori was supported in part by National Science Foundation grant DMS-1217122.  H.\ Shoori's current address: Caterpillar Technical Center, Chillicothe, IL 61523, USA.}%
	\and
	Hesam Shoori%
	\footnotemark[1]
	\and
	John~R.~Singler%
	\thanks{Department of Mathematics
		and Statistics, Missouri University of Science and Technology,
		Rolla, MO (\mbox{singlerj@mst.edu}, mhwdrc@mst.edu). J.~Singler was supported in part by National Science Foundation grant DMS-1217122.}
	\and
	Madhuka H.~Weerasinghe%
	\footnotemark[2]
}

\date{}

\begin{document}
\maketitle

\begin{abstract}
  We consider model order reduction of a nonlinear cable-mass system modeled by a 1D wave equation with interior damping and dynamic boundary conditions.  The system is driven by a time dependent forcing input to a linear mass-spring system at one boundary.  The goal of the model reduction is to produce a low order model that produces an accurate approximation to the displacement and velocity of the mass in the nonlinear mass-spring system at the opposite boundary.  We first prove that the linearized and nonlinear unforced systems are well-posed and exponentially stable under certain conditions on the damping parameters, and then consider a balanced truncation method to generate the reduced order model (ROM) of the nonlinear input-output system.  Little is known about model reduction of nonlinear input-output systems, and so we present detailed numerical experiments concerning the performance of the nonlinear ROM.  We find that the ROM is accurate for many different combinations of model parameters.
\end{abstract}
%

\section{Introduction}

Model order reduction (MOR) is currently a very active field of research in many disciplines with many potential applications including numerical simulation, optimization, uncertainty quantification, feedback control, and data assimilation; see, e.g., \cite{BennerSachsVolkwein14, bui2008parametric, daescu2007efficiency, fang2014reduced, gunzburger2016ensemble, gunzburger2005reduced, varshney2009feedback}.  MOR for linear differential equation systems with inputs and outputs is well established; however, little is known about MOR of nonlinear systems with inputs and outputs.

One main objective of this work is to understand the numerical performance of a type of balanced truncation model order reduction approach for a specific nonlinear PDE system with inputs and outputs.  Balanced truncation for linear input-output systems was first introduce by Moore in 1981 \cite{moore1981principal}, and is now a very popular model reduction approach \cite{Antoulas05, ZhouDoyleGlover96}.  The theory of balanced truncation model reduction for
nonlinear input-output systems was introduced later by Scherpen \cite{scherpen1993balancing}, but this method is not computationally feasible for large-scale systems. We consider another type of nonlinear balanced truncation model reduction that is closely related to balanced truncation for linear systems; specifically, the modes obtained from linear balanced truncation are used to reduce the nonlinear system via a Petrov-Galerkin projection.  This approach is computationally tractable and therefore has potential for various applications; however, there is no existing theoretical foundation for this MOR approach.

Due to this lack of theory, numerical studies are useful to test the performance of this MOR approach.  We are aware of only one detailed numerical study: in \cite{ilak2010model}, the authors numerically show that this nonlinear balanced truncation MOR approach is very effective for a 1D complex Ginzburg-Landau equation.

In this work, we consider the same model reduction approach for a nonlinear input-output cable-mass system that is represented by a one dimensional damped wave equation with dynamic
boundary conditions at both ends.  This model, which we introduce in Section \ref{sec:the_model}, was originally considered
as a heuristic model for a wave tank with a wave energy converter
\cite{Shoori14}.  We present detailed numerical experiments using the finite difference method and the balanced truncation MOR technique in Sections \ref{sec:BT_MOR} and \ref{sec:numerical_results}.

We believe the nonlinear cable-mass model considered here has not been explored elsewhere; therefore, we prove the well-posedness and exponential stability of the unforced linear and nonlinear models in Sections \ref{sec:linear_problem} and \ref{sec:nonlinear_problem}.  The well-posedness and exponential stability of many types of wave equation models with dynamic boundary conditions have been explored in the literature; see, e.g., \cite{burns1998reduced, conrad2001uniform, fourrier2013regularity, morgul1998stabilization, zhang2016stabilization} and the references therein.  The primary difference in the model considered here as compared to most of the models considered elsewhere is that the dynamic boundary conditions hold on all boundaries.  The paper \cite{conrad2001uniform} also considers a 1D wave equation with dynamic boundary conditions on all boundaries; however, the physical system considered in that work leads to very different boundary conditions than the ones we consider here.

MOR for wave equations has been discussed in the literature (see, e.g., \cite{AmsallemHetmaniuk14, batten2010reduced, GongWangWang17, herkt2013convergence, HuynhKnezevicPatera11, PengMohseni16, pereyra2016model}), however many existing works do not consider input-output model reduction as we do here.  The work \cite{batten2010reduced} also considers input-output types of model reduction for a different cable-mass model; however, that work explores the effectiveness of the model reduction for feedback control applications.  Feedback control of other PDE models with input in dynamic boundary conditions has also been explored in other works (see, e.g., \cite{BurnsCliff14, BurnsZietsman12, BurnsZietsman16, King94, Morris02}), however we do not believe model reduction has been explored in depth for such systems.

We also note that a preliminary version of this work appeared in \cite{battenmodel}; in this version, we consider a wider class of interior damping mechanisms, give more complete theoretical results, and provide more detailed numerical experiments.  Furthermore, after the conference paper \cite{battenmodel} was published, we discovered and corrected an error in our model reduction code.  The new computational experiments presented here indicate the MOR technique is far more accurate than reported in \cite{battenmodel}.

\section{The Model}
\label{sec:the_model}

We consider a flexible cable with mass-spring systems attached to each
end. Figure \ref{fig:cable_mass} illustrates the cable-mass system of interest. Each mass-spring system is connected to a rigid horizontal support.  The dotted line represents
the equilibrium position of the system.  Let $ w_0(t) $, $ w(t,x) $, and $ w_l(t) $ denote the position below equilibrium of the left mass (at location $ x = 0 $), the cable at location $ x $, and the right mass (at location $ x = l $), respectively, at time $ t $.  We assume the system is driven by an external force acting on the left mass-spring system, and that there are no other external forces.

\begin{figure}[htb]
\centering
\scalebox{1}{\includegraphics{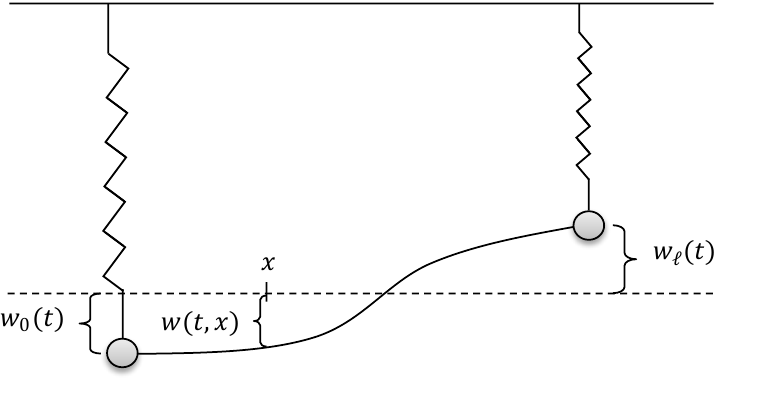}}
\caption{\label{fig:cable_mass}The cable mass system}
\end{figure}

We model the motion of the flexible cable with a damped 1D wave equation
on $0<x<l$. We include both Kelvin-Voigt and viscous damping in the model. We model the mass-spring systems with damped second order oscillators.  The left mass-spring system includes a time dependent external force input $u(t)$, and the right mass-spring system includes a nonlinear stiffening force.  This gives a wave equation with dynamic boundary conditions:
\begin{subequations}\label{eq:(1)}
\begin{align}
  w_{tt}(t,x)+\alpha w_{t}(t,x) &= \gamma w_{txx}(t,x)+\beta^{2}w_{xx}(t,x),\label{eq:(1a)}\\
  m_{0}\ddot{w}_{0}(t)+\alpha_{0}\dot{w}_{0}(t)+k_{0}w_{0}(t) &= \big( \gamma w_{tx}(t,0)+\beta^{2}w_{x}(t,0) \big)+u(t),\label{eq:(1b)}\\
  m_{l}\ddot{w}_{l}(t)+\alpha_{l}\dot{w}_{l}(t)+k_{l}w_{l}(t) &= \big( -\gamma w_{tx}(t,l)-\beta^{2}w_{x}(t,l) \big) - k_{3}\left[w_{l}(t)\right]^{3}.\label{eq:(1c)}
\end{align}
\end{subequations}
Each term in parenthesis in the dynamic boundary conditions is the force of the cable acting on the mass.  Here, $\gamma$ is the Kelvin-Voigt damping parameter, $\alpha$, $\alpha_{0}$, $\alpha_{l}$ are viscous damping parameters, $ m_0 $ and $ m_l $ are the masses, and $ k_0 $, $ k_l $, and $ k_3 $ are the stiffness parameters.  In the model the damping parameters are nonnegative, and the wave equation parameter $ \beta $ as well as the mass and stiffness parameters are all positive.  Finally, the position of the cable at each boundary must equal the position of each mass, therefore we have the \textit{displacement compatibility condition}
\begin{equation}\label{eqn:displacement_cc}
  w(t,0)=w_{0}(t),  \quad  w(t,l)=w_{l}(t).
\end{equation}

For the model reduction problem, we assume we have two system outputs:
the position and the velocity of the right mass, i.e.,
\[
y_{1}(t)=w_{l}(t),  \quad  y_{2}(t)=\dot{w}_{l}(t).
\]

\subsection{The Energy Function}
\label{subsec:energy}

Next, we give a preliminary investigation of the change in energy of the unforced system, i.e., the system with $ u(t) = 0 $.  This will help us obtain the correct inner products for an abstract formulation of the system.  Later we prove the energy decays to zero exponentially fast under certain assumptions on the system parameters.

Assume the solution of the above system is sufficiently smooth, and define the total kinetic energy of the cable as
\[
  E_{C,K}(t) = \frac{1}{2}\int_{0}^{l}w_{t}^{2}\,dx.
\]
Differentiating with respect to time and using the wave equation (\ref{eq:(1a)}) gives
\begin{align*}
  \frac{dE_{C,K}}{dt} &= \frac{1}{2}\int_{0}^{l}2w_{t}w_{tt}\,dx\\
    &= \int_{0}^{l} w_{t}(t,x) \, \big( \gamma w_{txx}(t,x)+\beta^{2}w_{xx}(t,x)-\alpha w_{t}(t,x) \big) \,dx.
\end{align*}
Integrate by parts to obtain
\begin{align*}
\frac{dE_{C,K}}{dt}= & -\gamma\int_{0}^{l}\left(w_{tx}(t,x)\right)^{2}dx-\beta^{2}\int_{0}^{l}w_{x}(t,x)w_{tx}(t,x)dx-\alpha\int_{0}^{l}\left(w_{t}(t,x)\right)^{2}dx\\
 & \,\,\,\,\,\,\,\,\,\,+w_{t}(t,l)\left[\gamma w_{tx}(t,l)+\beta^{2}w_{x}(t,l)\right]-w_{t}(t,0)\left[\gamma w_{tx}(t,0)+\beta^{2}w_{x}(t,0)\right].
\end{align*}
Using the boundary conditions and the displacement compatibility condition (\ref{eqn:displacement_cc}) gives
\begin{align*}
\frac{dE_{C,K}}{dt}= & -\gamma\int_{0}^{l}\left(w_{tx}(t,x)\right)^{2}dx-\beta^{2}\int_{0}^{l}w_{x}(t,x)w_{tx}(t,x)dx-\alpha\int_{0}^{l}\left(w_{t}(t,x)\right)^{2}dx\\
 & \,\,\,\,\,\,\,\,\,\,-\dot{w}_{l}(t)\left[m_{l}\ddot{w}_{l}(t)+\alpha_{l}\dot{w}_{l}(t)+k_{l}w_{l}(t)+k_{3}\left[w_{l}(t)\right]^{3}\right]\\
 & \,\,\,\,\,\,\,\,\,\,-\dot{w}_{0}(t)\left[m_{0}\ddot{w}_{0}(t)+\alpha_{0}\dot{w}_{0}(t)+k_{0}w_{0}(t)\right].
\end{align*}
This can be rewritten as
\begin{gather*}
  \frac{d}{dt} \Bigg[ \frac{1}{2}\int_{0}^{l}w_{t}^{2}\,dx\,+\frac{m_{l}}{2}\left(\dot{w}_{l}(t)\right)^{2}+ \frac{m_{0}}{2}\left(\dot{w}_{0}(t)\right)^{2}+\,\frac{\beta^{2}}{2}\int_{0}^{l}w_{x}^{2}\,dx\\
  \quad +\frac{k_{l}}{2}\left(w_{l}(t)\right)^{2}+\frac{k_{0}}{2}\left(w_{0}(t)\right)^{2}+\frac{k_{3}}{4}\left(w_{l}(t)\right)^{4} \Bigg] = \\
  \qquad -\gamma\int_{0}^{l} w_{tx}^{2} \, dx - \alpha\int_{0}^{l} w_{t}^{2} \,dx - \alpha_{0}\left(\dot{w}_{0}(t)\right)^{2} - \alpha_{l}\left(\dot{w}_{l}(t)\right)^{2}.
\end{gather*}

This suggests defining the system kinetic energy and potential energy as
\begin{align*}
E_{K}= & \int_{0}^{l}\frac{1}{2}w_{t}^{2}\,dx\,+\frac{m_{l}}{2}\left(\dot{w}_{l}(t)\right)^{2}+\frac{m_{0}}{2}\left(\dot{w}_{0}(t)\right)^{2},\\
E_{P}= & \int_{0}^{l}\frac{\beta^{2}}{2}w_{x}^{2}\,dx+\frac{k_{l}}{2}\left(w_{l}(t)\right)^{2}+\frac{k_{0}}{2}\left(w_{0}(t)\right)^{2}+\frac{k_{3}}{4}\left(w_{l}(t)\right)^{4}.
\end{align*}
These energy expressions can also be obtained by considering the kinetic
energy and potential energy of each component of the system.

The above energy equation gives
\begin{align}
\frac{d}{dt}E &= \frac{d}{dt}\left(E_{K}+E_{P}\right)\nonumber \\
  &= -\left[\gamma\int_{0}^{l} w_{tx}^{2} \,dx + \alpha\int_{0}^{l} w_{t}^{2} \, dx + \alpha_{0}\left(\dot{w}_{0}(t)\right)^{2} + \alpha_{l}\left(\dot{w}_{l}(t)\right)^{2}\right]\nonumber
\end{align}
Therefore, we have $\dot{E}(t)\leq0$.
%

\subsection{Variational Form}
\label{subsec:var_form}

In this subsection, we introduce the variational or weak form of
the system. Later we use this form to analyze the model.  Assume the solution $ [w,w_0,w_l] $ is smooth and satisfies the displacement compatibility condition (\ref{eqn:displacement_cc}).  Multiply the wave equation (\ref{eq:(1a)}) by a smooth test function $ h = h(x) $ satisfying $ h(0) = h_0 $ and $ h(l) = h_l $ and integrate by parts to obtain
\begin{align*}
  \int_{0}^{l}w_{tt}\,h\,dx+\alpha\int_{0}^{l}w_{t}\,h\,dx-h_{l}\left[\gamma w_{tx}(l)+\beta^{2}w_{x}\left(l\right)\right]+h_{0}\left[\gamma w_{tx}(0)+\beta^{2}w_{x}\left(0\right)\right]\\
  +\gamma\int_{0}^{l}w_{tx}\,h_{x}\,dx+\beta^{2}\int_{0}^{l}w_{x}\,h_{x}\,dx=0.
\end{align*}
As in the above energy argument, we use the boundary conditions to give the variational form
\begin{gather}
  \int_{0}^{l}w_{tt}\,h\,dx+m_{l}\ddot{w}_{l}(t)h_{l}+m_{0}\ddot{w}_{0}(t)h_{0}+ \beta^{2}\int_{0}^{l}w_{x}\,h_{x}\,dx+k_{l}w_{l}(t)h_{l}+k_{0}w_{0}(t)h_{0}\nonumber \\
  \quad +\int_{0}^{l}\left[\alpha w_{t}\,h\,+\gamma w_{tx}\,h_{x}\,\right]dx+h_{l}\alpha_{l}\dot{w}_{l}(t)+h_{0}\alpha_{0}\dot{w}_{0}(t)+ k_{3} h_l \left[w_{l}(t)\right]^{3}=0.\label{eq:variational formula one vertion}
\end{gather}

Now we give details about the function spaces to make the weak formulation precise.  Let $H$ be the real Hilbert space $H=L^{2}(0,l)\times\mathbb{R}^{2}$
with the inner product of $z=\left[w,w_{0},w_{l}\right]\in H$~and
$\psi=\left[p,\,p_{0},\,p_{l}\right]\in H$ defined by
\begin{equation}
  (z,\psi)_{H}=\int_{0}^{l}w \,p \, dx + m_{0} \, w_{0} \, p_{0} + m_{l} \, w_{l} \,p_{l}.\label{eq:Hnorm-1}
\end{equation}
Let $V\subset H$ be the set of elements $z=\left[w,w_{0},w_{l}\right]\in H^{1}(0,l)\times\mathbb{R}^{2}$
satisfying the displacement compatibility condition $w(0)=w_{0}$ and $w(l)=w_{l}$.  For $z\in V$ as above and $\psi=\left[p,\,p_{0},\,p_{l}\right]\in V$
define the $V$ inner product of $z$ with $\psi$ by
\begin{equation}
  (z,\psi)_V = \int_{0}^{l}\beta^{2} \, w_{x} \,p_{x} \,dx + k_{0} \, w_{0} \, p_{0} + k_{l} \, w_{l} \,p_{l}.\label{eq:vnorm-1}
\end{equation}
We also use the notation $ \sigma_{1}(z,\psi) = (z,\psi)_V $.

The $ H $ and $ V $ inner products, (\ref{eq:Hnorm-1}) and (\ref{eq:vnorm-1}), can
be derived from the energy function; the $H$ and $V$ norms are directly related to the system kinetic and potential energies,
respectively.  Specifically,
\[
  E_K = \frac{1}{2} ( z_t,z_t )_H = \frac{1}{2} \| z_t \|_H^2,  \quad  E_P = \frac{1}{2} ( z,z )_V + \frac{k_3}{4} w_l^4 = \frac{1}{2} \| z \|_V^2  + \frac{k_3}{4} w_l^4.
\]
Furthermore, both inner products appear in the variational form (\ref{eq:variational formula one vertion}).

The Gelfand triple is $V\hookrightarrow H\hookrightarrow V'$ with
pivot space $H$ and the algebraic dual of $V$ is $V'$. We define
$\langle g,v\rangle$ for $g\in V',v\in V$ as $\langle g,v\rangle=g(v)$.  Note if $ g \in H $ and $ v \in V $, then $ \langle g, v \rangle = (g,v)_H $.  Also, we define the damping bilinear form $\sigma_{2}:\,V\times V\rightarrow\mathbb{R}$
\begin{equation}
  \sigma_{2}(z,\psi)=\int_{0}^{l}(\gamma \, w_{x} \, p_{x} + \alpha \, w \, p) \,dx + \alpha_{0} \, w_{0} \, p_{0}+\alpha_{l} \, w_{l} \, p_{l}. \label{eq:sigma2norm}
\end{equation}
Note that this bilinear form occurs in the variational form (\ref{eq:variational formula one vertion}) as a damping term with all first order time derivatives.

The spaces and inner products are motivated by the above variational
form (\ref{eq:variational formula one vertion}). Further we can rewrite
the above variational form (\ref{eq:variational formula one vertion}) as
\begin{equation}
\langle z_{tt},\psi\rangle+\sigma_{1}(z(t),\psi)+\sigma_{2}(z_{t},\psi)+\left(f(z),\psi\right)_{H}=0,\label{eq:variation formula full}
\end{equation}
where $f(z)=[0,0,k_{3} m_l^{-1} w_{l}^{3}]$ is the
nonlinear term.

\section{The Linear Problem}
\label{sec:linear_problem}

We begin by analyzing the variational form for the linear problem
\[
\langle z_{tt},\psi\rangle+\sigma_{1}(z(t),\psi)+\sigma_{2}(z_{t},\psi)=0.
\]
We prove the linear problem is well-posed, and also exponentially stable under certain assumptions on the damping parameters.  The exponential stability is necessary for the application of the balanced truncation model reduction technique considered later.  Some of the results and proofs in this section are given in the preliminary version of this work \cite{battenmodel}; we reproduce them for completeness.

\subsection{Function Spaces}

We first present basic results about the function spaces that we frequently use in this work.
\begin{lemma}\label{lemma:V_properties}
The space $V$ with the above inner product (\ref{eq:vnorm-1}) is
a real Hilbert space, $V$ is dense in $H$, and $V$ is separable.
\end{lemma}
The proof of Lemma \ref{lemma:V_properties} is given in the Appendix.

We use the inequalities in the following lemma to prove the well-posedness and exponential stability of the system.
\begin{lemma}\label{lem:lemma to show inequalyties}
If $z=[w,w_{0},w_{l}]\in V$, then
\begin{align}
  \left|w(x)\right|^{2} &\leq 2 w_{0}^{2}+2l\left\Vert w_{x}\right\Vert _{L^{2}(0,l)}^{2},\label{eq:8-1}\\
  \left\Vert w\right\Vert _{L^{2}(0,l)}^{2} &\leq 2l \left[ w_{0}^{2}+l\left\Vert w_{x}\right\Vert_{L^{2}(0,l)}^{2}\right],\label{eq:9-1}\\
  w_{l}^{2} &\leq 2w_{0}^{2}+2l\left\Vert w_{x}\right\Vert_{L^{2}(0,l)}^{2}.\label{eq:10}\\
%
%
%
  w_{0}^{2} &\leq 2w_{l}^{2}+2l\left\Vert w_{x}\right\Vert _{L^{2}(0,l)}^{2}\label{eq:14}.
\end{align}
\end{lemma}
\begin{proof}
Since $ w \in H^1(0,l) $ and $ w(0) = w_0 $, we have
\[
w(x)=w_{0}+\int_{0}^{x}w_{\xi}(\xi)\,d\xi.
\]
Taking absolute values, using the triangle inequality, and then applying H\"{o}lder's inequality gives
$$
  \left|w(x)\right| \leq \left|w_{0}\right|+\int_{0}^{x}\left|w_{\xi}(\xi)\right|\,d\xi  \leq  \left|w_{0}\right|+l^{\frac{1}{2}}\left\Vert w_{x}\right\Vert _{L^{2}(0,l)}.
$$
Squaring this inequality and using Young's inequality gives (\ref{eq:8-1}); integrating (\ref{eq:8-1}) from $x=0$ to $x=l$ gives (\ref{eq:9-1}); and evaluating equation (\ref{eq:8-1}) at $x=l$ yields (\ref{eq:10}).

Using $ w(x)=w_{l}-\int_{x}^{l}w_{\xi}(\xi)\,d\xi $, the proof of (\ref{eq:14}) follows similarly.
\end{proof}

\begin{lemma}
V is continuously embedded in H. \label{Lemma V cts embedded}\end{lemma}
\begin{proof}
Let $ z = [w,w_0,w_l] \in V $.  We use the $H $ and $V$ inner products and the inequality (\ref{eq:9-1})
from Lemma \ref{lem:lemma to show inequalyties} to obtain
\begin{align}
\left\Vert z\right\Vert _{H}^{2} &= \int_{0}^{l}w^{2}dx+m_{0}w_{0}^{2}+m_{l}w_{l}^{2}\nonumber \\
  &= \left\Vert w\right\Vert _{L^{2}(0,l)}^{2}+m_{0}w_{0}^{2}+m_{l}w_{l}^{2}\nonumber \\
  &\leq 2l\left[\left|w_{0}\right|^{2}+l\left\Vert w_{x}\right\Vert _{L^{2}(0,l)}^{2}\right]+m_{0}w_{0}^{2}+m_{l}w_{l}^{2}\nonumber \\
  &\leq 2l^{2}\int_{0}^{l}w_{x}^{2}\,dx+(2l+m_{0})w_{0}^{2}+m_{l}w_{l}^{2}\nonumber \\
  &= \left(\frac{2l^{2}}{\beta^{2}}\right)\beta^{2}\int_{0}^{l}w_{x}^{2}\,dx+\left(\frac{2l+m_{0}}{k_{0}}\right)k_{0}w_{0}^{2}+\left(\frac{m_{l}}{k_{l}}\right)k_{l}w_{l}^{2}\nonumber \\
  &\leq C_3 \left[k_{0}w_{0}^{2}+k_{l}w_{l}^{2}+\beta^{2}\int_{0}^{l}w_{x}^{2}\,dx\right],\nonumber
\end{align}
where
$$
  C_3 = \max\left\{ \frac{2l+m_{0}}{k_{0}}, \frac{m_{l}}{k_{l}}, \frac{2l^{2}}{\beta^{2}} \right\}.
$$
This gives $ C_3^{-1} \| z \|_H^2 \leq \| z \|_V^2 $ for all $ v \in V $, and therefore $ V $ is continuously embedded in $ H $.
\end{proof}

\subsection{Well-Posedness and Exponential Stability}
\label{subsec:linear_theory}

To show the linear problem is well-posed, we rewrite the problem as $ \dot{x} = \mathcal{A} x $ and show $\mathcal{A}$ generates a $C_{0}$-semigroup on $\mathcal{H}=V\times H$.  We need the following basic concepts concerning bilinear forms acting on $ V $.
\begin{definition}
  A bilinear form $\sigma\,:\,V\times V\rightarrow\mathbb{R}$ is
  \begin{itemize}
    \item $V$-continuous if there exists $ c_1 > 0 $ such that $\left|\sigma(\varphi,\psi)\right|\leq c_{1}\left\Vert \varphi\right\Vert _{V}\left\Vert \psi\right\Vert _{V}$ for all $\varphi$ and $\psi$ in $V$;
    \item $V$-elliptic if there exists a constant $c_{2}>0$ such that $\sigma\left(\varphi,\varphi\right)\geq c_{2}\left\Vert \varphi\right\Vert_{V}^{2}$ for all $\varphi$ in $V$;
    \item $H$-semielliptic if there exists a constant $c_{3}\geq0$ such that $\sigma\left(\varphi,\varphi\right)\geq c_{3}\left\Vert \varphi\right\Vert_{H}^{2}$ for all $\varphi$ in $V$; and also $\sigma$ is $H$-elliptic if $c_{3}>0$.
  \end{itemize}
\end{definition}

We follow the presentation in \cite[Section 8.1]{banks2012functional} in order to find the linear operator $\mathcal{A}$.  First, Lemma (\ref{lem:lemma to show inequalyties}) can be used to show $\sigma_{2}$ is $V$-continuous.  Since $\sigma_{1} $ and $\sigma_{2}$ are $ V $-continuous we have that there exists
operators $A_{i}\in\mathcal{L}(V,V')$ for $i=1,2$ such that
\[
  \sigma_{i}(\varphi,\psi)=\langle A_{i}\varphi,\psi\rangle  \quad  \mbox{for all $ \varphi,\psi\in V $.}
\]
Define the operator $ \mathcal{A} : D(\mathcal{A}) \subset \mathcal{H} \to \mathcal{H} $ by
\[
  D(\mathcal{A})=\left\{ x=[\varphi,\psi]\in\mathcal{H}\,:\,\psi\in V, \,\, A_{1}\varphi+A_{2}\psi\in H\right\}
\]
and
\begin{equation}
\mathcal{A}=\left[\begin{array}{cc}
0 & I\\
-A_{1} & -A_{2}
\end{array}\right].\label{eq:operator A}
\end{equation}

\begin{theorem}
The operator $\mathcal{A}$ defined above generates a $C_{0}$-semigroup
on $\mathcal{H}=V\times H$.\label{thm:A generates c0 semigroup}\end{theorem}
\begin{proof}
Due to the properties of the spaces $ H $ and $ V $, the result follows directly from Theorem 8.2 in \cite{banks2012functional} since $ \sigma_1 $ is the $ V $ inner product and $ \sigma_2 $ is $ H $-semielliptic.
\end{proof}
Since $\mathcal{A}$ generates a $C_{0}$-semigroup $T(t)$ on $\mathcal{H}=V\times H$,
we have $T(t)x_{0}$ is the unique solution of $\dot{x}=Ax$ where $x(0)=x_{0}$.

For the exponential stability of the problem, we restrict our attention to the model with interior damping; i.e., the Kelvin-Voigt damping parameter $ \gamma $ is positive or the viscous damping parameter $ \alpha $ is positive.  In this case, the easiest way to prove exponential stability is to show $ \sigma_2 $ is $ H $-elliptic or $ V $-elliptic.  Note that since $ V $ is continuously embedded in $ H $, if $ \sigma_2 $ is $ V $-elliptic then it must also be $ H $-elliptic; additionally, if $ \sigma_2 $ is $ V $-elliptic then the semigroup is also analytic.
\begin{theorem}
If $ \sigma_2 $ is $ H $-elliptic, then the operator $\mathcal{A}$ defined in (\ref{eq:operator A}) is the infinitesimal
generator of an exponentially stable $C_{0}$-semigroup $T(t)$ on $\mathcal{H}=V\times H$.  Furthermore, if $ \sigma_2 $ is $ V $-elliptic, then $T(t)$ is exponentially stable and also analytic.\label{thm:exponential stable}\end{theorem}
\begin{proof}
This follows directly from Theorems 8.1 and 8.3 in \cite{banks2012functional}.
\end{proof}
In this work, we restrict our analysis to the cases where the damping bilinear form $\sigma_{2}$ is $H$-elliptic or $ V $-elliptic.  The analysis of exponential stability for the model when this condition is not satisfied is more involved; we leave the analysis of such cases to be considered elsewhere.  We prove exponential stable for the linear system for three main examples of damping parameter sets.  We also consider model reduction computations for two other examples in the numerical results.
\begin{description}
\item [{Example 1}] $\gamma,\alpha_{l}>0$ and $\alpha_{0}=\alpha=0$
\end{description}
We first consider the case of Kelvin-Voigt damping ($ \gamma > 0 $) and viscous damping in the right mass-spring system ($ \alpha_l > 0 $).  We prove $\sigma_{2}$ is $V$-elliptic.
Let $ z = [w,w_0,w_l] \in V $.  Using the inequality (\ref{eq:14}) gives
\begin{align*}
  \left\Vert z\right\Vert_{V}^{2} &= \int_{0}^{l}\beta^{2}w_{x}^{2}dx+k_{0}w_{0}^{2}+k_{l}w_{l}^{2}\\
    &\leq (\beta^{2}+2 l k_0)\int_{0}^{l}w_{x}^{2}dx+(k_{l}+2 l k_0)w_{l}^{2}\\
    &= \left(\frac{\beta^{2}+2lk_0}{\gamma}\right)\int_{0}^{l}\gamma w_{x}^{2}dx+\left(\frac{k_{l}+2lk_0}{\alpha_{l}}\right)\alpha_{l}w_{l}^{2}\\
    &\leq  C \sigma_2(z,z),
\end{align*}
where $ C = \max\left\{ \frac{\beta^{2}+2lk_0}{\gamma}, \frac{k_{l}+2lk_0}{\alpha_{l}} \right\} $.  Therefore, we have $ C^{-1} \| z \|_V^2  \leq  \sigma_2(z,z) $ for all $ z \in V $, i.e., $ \sigma_2 $ is $ V $-elliptic.

It can also be shown that $ \sigma_2 $ is $ V $-elliptic in the similar case when $\gamma,\alpha_{0}>0$ and $\alpha,\alpha_{l}=0$.

\begin{description}
\item [{Example 2}] $\gamma=0$ and $\alpha,\alpha_{0},\alpha_{l}>0$
\end{description}
Next, we consider the case of viscous damping in the wave equation and both mass-spring systems ($ \alpha, \alpha_0, \alpha_l > 0 $).  Let $ z = [w,w_0,w_l] \in V $.  Since
\begin{align*}
\sigma_{2}(z,z)= & \int_{0}^{l}\alpha w^{2}dx+\alpha_{0}w_{0}^{2}+\alpha_{l}w_{l}^{2},\\
\left\Vert z\right\Vert _{H}^{2}= & \int_{0}^{l}w^{2}dx+m_{0}w_{0}^{2}+m_{l}w_{l}^{2},
\end{align*}
it is clear that $\sigma_{2}$ is $H$-elliptic in this case.
\begin{description}
\item [{Example 3}] $\gamma,\alpha>0$ and $\alpha_{0},\alpha_{l}=0$
\end{description}
In this last case, we consider both Kelvin-Voigt and viscous damping in the interior, but no damping in either boundary.  We prove $ \sigma_2 $ is $ V $-elliptic.

We rewrite the bilinear form of $\sigma_{2}$ and the $V$ inner products
according to the above parameters:
\begin{align*}
\sigma_{2}(z,z) &= \int_{0}^{l}(\gamma w_{x}^{2}+\alpha w^{2})\,dx,\\
\left\Vert z\right\Vert _{V}^{2} &= \int_{0}^{l}\beta^{2}w_{x}^{2} \, dx+k_{0}w_{0}^{2}+k_{l}w_{l}^{2}.
\end{align*}
Recall $ L^\infty(0,l) $ is continuously embedded in $ H^1(0,l) $, and so there exists a constant $ C > 0 $ such that
$$
  \| w \|^2_{L^\infty(0,l)}  \leq  C \| w \|_{H^1(0,l)}^2 = C \int_{0}^{l}(w^{2}+w_{x}^{2})\,dx.
$$
Therefore, since $ w_0 = w(0) $ and $ w_l = w(l) $,
\begin{align*}
  \left\Vert z\right\Vert _{V}^{2} &\leq \int_{0}^{l}\beta^{2}w_{x}^{2}\,dx+k_{0}\left\Vert w\right\Vert _{L^\infty}^{2}+k_{l}\left\Vert w\right\Vert _{L^\infty}^{2}\\
 &\leq \int_{0}^{l}\beta^{2}w_{x}^{2}\,dx+(k_{0}+k_{l})C\int_{0}^{l}(w^{2}+w_{x}^{2})\,dx\\
 &\leq C_{3}\,\sigma_{2}(z,z),
\end{align*}
where
$$
  C_{3} = \max\left\{ \frac{\beta^2+(k_{0}+k_{l})C}{\gamma}, \frac{(k_{0}+k_{l})C}{\alpha}\right\}.
$$
Therefore, $ C_3^{-1} \left\Vert z\right\Vert_{V}^{2} \leq \sigma_{2}(z,z) $, i.e., $ \sigma_2 $ is $ V $-elliptic.
%

\section{The Nonlinear Problem}
\label{sec:nonlinear_problem}

%
Next, we analyze the well-posedness and exponential stability of the full unforced nonlinear problem.  First, we write the nonlinear problem as
\begin{equation}
  \dot{x}(t)=\mathcal{A}x(t)+\mathcal{F}(x(t)),  \quad  x(0) = x_0,\label{eq:nonlinear problem}
\end{equation}
on $\mathcal{H}=V\times H$ where the linear operator $\mathcal{A}$ is defined in Section \ref{subsec:linear_theory} and the nonlinear term $\mathcal{F}\,:\,\mathcal{H}\rightarrow\mathcal{H}$
is defined for $ x = [ \varphi, \psi ] \in \mathcal{H} $ with $ \varphi = [w,w_0,w_l] \in V $ by
\[
\mathcal{F}(x)=\left[\begin{array}{c}
0\\
F_{0}(\varphi)
\end{array}\right],  \quad  F_{0}(\varphi)=\left[\begin{array}{c}
0\\
0\\
m_{l}^{-1}k_{3} w_{l}^{3}
\end{array}\right].
\]

\begin{theorem}
The nonlinear cable mass system has unique mild solution on some time interval $ [0,t^*) $.\end{theorem}
\begin{proof}
It can be checked that the nonlinear term $\mathcal{F}$ is locally Lipschitz continuous
on $\mathcal{H}$.  Therefore, the result follows using semigroup theory; see, e.g., \cite[Theorem 1.4 in section 6.1]{pazy2012semigroups}.
\end{proof}


Next, we prove the unforced nonlinear system is exponential stability when the damping bilinear form $\sigma_{2}$ is $H$-elliptic.  We gave examples of damping parameters that guarantee $ \sigma_2 $ is either $ H $-elliptic or $ V $-elliptic in the previous section.  Recall also that if $ \sigma_2 $ is $ V $-elliptic, then it is also $ H $-elliptic.  For the proof, we use the energy argument from Section \ref{subsec:energy}, the variational formulation in Section \ref{subsec:var_form}, and also the following lemma.

\begin{lemma}[Theorem 8.1 in \cite{komornik1994exact}]\label{thm:exponential stability of nonlinear system}
Let $E\,:\,\mathbb{R}^{+}\rightarrow\mathbb{R}^{+}$ be a non-increasing function.  If there exists a constant
$T>0$ such that $\int_{s}^{\infty}E(t)\,\leq TE(s) $ for all $ s \geq 0 $, then
$E(t)\leq E(0)e^{1-t/T}$ for all $t\geq0$.
\end{lemma}
\begin{theorem}
If $\sigma_{2}$ is $H-$elliptic and the solution $ x = [z, z_t] $, with $ z = [w,w_0,w_l] $, of the unforced nonlinear cable-mass problem (\ref{eq:nonlinear problem}) is sufficiently smooth, then the energy $E(t)=\frac{1}{2}\left\Vert z_{t}\right\Vert _{H}^{2}+\frac{1}{2}\left\Vert z\right\Vert _{V}^{2}+\frac{k_{3}}{4}\left[w_{l}(t)\right]^{4}$
of the solution with initial data $ x(0) = x_0 \in \mathcal{H} $ decays exponentially fast as $t\rightarrow\infty.$
\end{theorem}
\begin{remark}
  It may be possible to identify conditions on the damping parameters or the initial data $ x_0 \in \mathcal{H} $ that provide the required smoothness of the solution for the proof of this exponential stability result.  We leave this to be considered elsewhere.
\end{remark}
\begin{proof}
First, since the solution is sufficiently smooth, the energy argument from Section \ref{subsec:energy} gives $ E'(t) \leq 0 $, where
$$
  E(t) = E_K(t) + E_P(t),  \quad  E_K = \frac{1}{2} \| z_t \|_H^2,  \quad  E_P = \frac{1}{2} \| z \|_V^2 + \frac{k_{3}}{4} \left[ w_{l} \right]^{4}.
$$
Therefore, $ E(t) $ is non-increasing.  Also, $ \frac{1}{2} \| x \|_\mathcal{H}^2  =  \frac{1}{2} \| z \|_V^2 + \frac{1}{2} \| z_t \|_H^2  \leq  E(t) $.  Since $ E(t) $ is bounded, $ \| x(t) \|_\mathcal{H}^2 $ cannot blow up in finite time; therefore, semigroup theory gives that the solution must exist for all $ t > 0 $ \cite[Theorem 1.4 in section 6.1]{pazy2012semigroups}.

Next, we show the energy function satisfies the remaining condition in the above lemma by separately considering the kinetic and potential energies.
Our proof uses ideas from the proof of Theorem 3.2 in Fourrier and Lasiecka's work \cite{fourrier2013regularity}.  To use the above lemma, let $ s \geq 0 $ and $ t > s $.

\textbf{Step 1:} First we consider the kinetic energy.  Recalling the energy argument from Section \ref{subsec:energy} immediately gives
\[
  E'(t)=-\sigma_{2}(z_{t},z_{t}).
\]
Integrate with respect to time from $s$ to $t$ to obtain
\[
  E(t)=E(s)-\int_{s}^{t}\sigma_{2}(z_{t},z_{t}) \, d\tau.
\]
Since $ \sigma_2 $ is $ H $-elliptic, there is a constant $ C > 0 $ such that $ \sigma_{2}(z_{t},z_{t}) \geq (C/2) \| z_t \|_H^2 = C E_K $.  Therefore, for $ C_1 = C^{-1} $,
$$
  \int_{s}^{t}E_{K}(\tau)\,d\tau \leq C_1 E(s) - C_1 E(t)  \leq  C_1 E(s),
$$
since $ E(t) \geq 0 $.

\textbf{Step 2:} Next, we consider the potential energy.  Substitute $ \psi = z = [w,w_0,w_l] $ in the variational formulation (\ref{eq:variation formula full}) to give
\[
  (z_{tt},z)_{H} + (z,z)_V+\sigma_{2}(z_{t},z) + \left(f(z),z\right)_{H}=0.
\]
Since $ \sigma_2 $ is a symmetric bilinear form, we have $ \sigma_{2}(z_{t},z) = \frac{1}{2} \frac{d}{dt} \sigma_2(z,z) $.  Integrate with respect to time from $ s $ to $ t $, and then integrate by parts in time to obtain
\begin{gather*}
(z_{t}(t),z(t))_{H}-(z_{t}(s),z(s))_{H}-\int_{s}^{t}\left(z_{t},z_{t}\right)_{H}d\tau+\int_{s}^{t}(z,z)_{V}d\tau\\
+\frac{1}{2} \sigma_{2}(z(t),z(t)) - \frac{1}{2} \sigma_{2}(z(s),z(s))+\int_{s}^{t}k_{3}(w_{l}(t))^{4}d\tau = 0.
\end{gather*}
Using the definition of kinetic energy and potential energy gives
\begin{gather*}
\frac{1}{2}\sigma_{2}(z(t),z(t))+2\int_{s}^{t}E_{P}(\tau)d\tau + \int_{s}^{t}\frac{k_{3}}{2}\left[w_l(t)\right]^{4}d\tau - 2\int_{s}^{t}E_{K}(\tau)d\tau\\
 +(z_{t}(t),z(t))_{H}-(z_{t}(s),z(s))_{H}-\frac{1}{2}\sigma_{2}(z(s),z(s)) = 0.
\end{gather*}
We remove the nonnegative term $\int_{s}^{t}\frac{k_{3}}{2}\left[w_l(t)\right]^{4}d\tau \geq 0$, and the equality becomes the inequality
\begin{align*}
\frac{1}{2}\sigma_{2}(z(t),z(t))+2\int_{s}^{t}E_{P}(\tau)d\tau &\leq 2\int_{s}^{t}E_{K}(\tau)d\tau-(z_{t}(t),z(t))_{H}\\
 & \quad +(z_{t}(s),z(s))_{H}+\frac{1}{2}\sigma_{2}(z(s),z(s)).
\end{align*}
Use $(u,v)_{H}\leq\left\Vert u\right\Vert _{H}\left\Vert v\right\Vert _{H}$ and the $ V $-continuity of $\sigma_{2}$ to obtain
\begin{align*}
  &\frac{1}{4}\sigma_{2}(z(t),z(t))+\int_{s}^{t}E_{P}(\tau)d\tau\\
  &\quad  \leq \int_{s}^{t}E_{K}(\tau)d\tau + \frac{1}{2} \left[\left\Vert z_{t}(s)\right\Vert _{H}\left\Vert z(s)\right\Vert _{H}+\left\Vert z_{t}(t)\right\Vert _{H}\left\Vert z(t)\right\Vert _{H}+\frac{C_{2}}{2}\left\Vert z(s)\right\Vert _{V}^{2}\right].
\end{align*}
Next, use the result from Step 1, Young's inequality, and the continuous embedding $ V \hookrightarrow H $ to obtain
\begin{align*}
  & \frac{1}{4}\sigma_{2}(z(t),z(t))+\int_{s}^{t}E_{P}(\tau)d\tau\\
  & \quad \leq C_{1} E(s) + \frac{1}{2} \left[\frac{1}{2}\left\Vert z_{t}(s)\right\Vert_{H}^{2} + \frac{1}{2}\big( C_{3} + C_2 \big) \left\Vert z(s)\right\Vert_{V}^{2} \right]\\
  &  \qquad  + \frac{1}{2} \left[ \frac{1}{2}\left\Vert z_{t}(t)\right\Vert _{H}^{2} + \frac{C_{3}}{2}\left\Vert z(t)\right\Vert _{V}^{2} \right]\\
  & \quad \leq C_{1} E(s) + \frac{1}{2} \left[\frac{1}{2}\left\Vert z_{t}(s)\right\Vert_{H}^{2} + \frac{1}{2}\big( C_{3} + C_2 \big) \left\Vert z(s)\right\Vert_{V}^{2} + \frac{k_{3}}{4}\left[w_l(s)\right]^{4} \right]\\
  &  \qquad  + \frac{1}{2} \left[ \frac{1}{2}\left\Vert z_{t}(t)\right\Vert _{H}^{2} + \frac{C_{3}}{2}\left\Vert z(t)\right\Vert _{V}^{2} + \frac{k_{3}}{4}\left[w_l(t)\right]^{4} \right]\\
  &  \quad  \leq  C_{1} E(s) + C_4 E(s) + C_5 E(t),
\end{align*}
where
$$
  C_4 = \frac{1}{2} \max\{ C_3 + C_2, 1 \},  \quad  C_5 = \frac{1}{2} \max\{ C_3, 1 \}.
$$
Since $ t > s $ and $ E'(t) \leq 0 $ for all $ t \geq 0 $, we have $ E(t) \leq E(s) $.  This gives
$$
  \int_{s}^{t}E_{P}(\tau)d\tau  \leq  C_6 E(s),  \quad  C_6 = C_1 + C_4 + C_5.
$$
Combining this result with the result of Step 1, letting $ t \to \infty $, and using Lemma \ref{thm:exponential stability of nonlinear system} proves the result.
\end{proof}

\section{Balanced Truncation Model Reduction}
\label{sec:BT_MOR}

Next, we return to the forced nonlinear cable-mass system (\ref{eq:(1)}) with the system input $ u(t) $ in the dynamic boundary condition (\ref{eq:(1b)}) and system output
$$
  y(t) = [ w_l(t), \dot{w}_l(t) ]^T
$$
of the position and velocity of the right mass.  In this section, we describe a balanced truncation model reduction approach for this nonlinear system.  We begin by briefly reviewing balanced truncation model reduction for linear input-output ordinary differential equation systems and then infinite dimensional systems.  We outline the finite difference method we use to approximate the nonlinear cable-mass model, and then describe the balanced truncation model reduction method for the approximating nonlinear finite dimensional system.

\subsection{Finite Dimensional Balanced Truncation Theory}

Balanced truncation is one of the most popular model reduction
methods in control and systems theory, and it is based on the idea of controllability
and observability \cite{Antoulas05, ZhouDoyleGlover96}.  To review the main ideas, consider the exponentially stable linear time invariant dynamical
system in state space form
\begin{align}
\dot{x}(t) &= Ax(t)+Bu(t),\nonumber \\
y(t) &= Cx(t),\label{eq:first order system-1}
\end{align}
with $x(t) \in \mathbb{R}^{N}$ is the state, $u(t)\in\mathbb{R}^{m}$ is the input, and $y(t)\in\mathbb{R}^{p}$ is the output. Moreover,  $A \in \mathbb{R}^{N\times N}$, $B \in \mathbb{R}^{N\times m}$, and $C\in \mathbb{R}^{p\times N}$ are constant matrices, and $ A $ is stable.

To reduce the complexity of the system, we approximate the problem using a reduced number of states $ r \ll N $.  Balanced truncation produces a reduced order model
\begin{align}
\dot{a}(t) &= A_r a(t)+B_r u(t),\nonumber \\
y_r(t) &= C_r a(t),\label{eq:ROM-1}
\end{align}
where $a(t) \in \mathbb{R}^{r}$ is the reduced order state, such that the error in the output $ \| y(t) - y_r(t) \| $ is small when the same input $ u(t) $ is applied to both systems.

To do this, let $ T \in \mathbb{R}^{N\times N} $ be invertible, and make the change of variable $z=Tx$.  Then we can write the original system (\ref{eq:first order system-1}) as
\begin{align*}
\dot{z}(t) &= T^{-1}ATz(t)+T^{-1}Bu(t),\\
y(t) &= CTz(t).
\end{align*}
It can be checked that the transfer function $ G(s) = C (s I - A)^{-1} B $ relating inputs to outputs in the original system is equal to the transfer function of the transformed system.

Since $ A $ is stable, the controllability and observability Gramians, $P, Q \in \mathbb{R}^{N \times N}$, are the unique positive semidefinite solutions to the Lyapunov equations $AP+PA^{T}+BB^{T}=0$ and $A^{T}Q+QA+C^{T}C=0$.  It can be checked that the Gramians of the transformed system are given by $\hat{P}=TPT^{T}$ and $\hat{Q}=\left(T^{-1}\right)^{T}QT^{-1}$.  Furthermore, if $ P $ and $ Q $ are positive definite, there exists $ T $ such that transformed Gramians $ \hat{P} $ and $ \hat{Q} $ are \textit{balanced}, i.e., they are equal and diagonal; the positive diagonal entries are called the Hankel singular values of the system, and they are ordered from greatest to least.

The states in the transformed system corresponding to small Hankel singular values are truncated to produce the balanced low order model.  In addition, the truncation error between the transfer function $ G(s) $ of the original system and the transfer function $ G_r(s) = C_r (s I - A_r)^{-1} B_r $ of the balanced low order model can be bounded by
\begin{equation}\label{eqn:BT_error_bound}
  \left\Vert G(s)-G_{r}(s)\right\Vert_{\infty} \leq 2\sum_{i > r} \sigma_{i},
\end{equation}
where $ \{ \sigma_{i} \}_{i=1}^N $ are the ordered Hankel singular values of the system and the norm is the $ \mathcal{H}_\infty $ system norm.  Therefore, if the Hankel singular values decay to zero quickly, then the balanced low order model can provide a good approximation to the input-output response of the full order system.
%

\subsection{Infinite Dimensional Balance Truncation Theory}

Since we consider a partial differential equation system in this work, we briefly review balanced truncation model reduction for linear infinite dimensional systems of the form
\begin{align}
\dot{x}(t) &= \mathcal{A} x(t)+\mathcal{B} u(t),\nonumber \\
y(t) &= \mathcal{C}x(t),\label{eqn:inf_dim_linear_sys}
\end{align}
holding over a Hilbert space $ X $, where $ \mathcal{A} : D(\mathcal{A}) \subset  X \to X $ is the generator of an exponentially stable $ C_0 $-semigroup on $ X $, and $ \mathcal{B} : \mathbb{R}^m \to X $ and $ \mathcal{C} : X \to \mathbb{R}^p $ are both bounded linear operators.  We also verify the theory holds for the linear cable-mass system.

The theoretical background for the existence of the balanced truncation for this class of infinite dimensional linear systems can be found in \cite{curtain2001compactness, glover1988realisation}.  Specifically, there is a transformed system holding over the Hilbert space $ \ell^2 $ that is balanced, i.e., the controllability and observability Gramians are equal and diagonal.  Also, as in the finite dimensional case, the diagonal entries are called the Hankel singular values and are ordered from greatest to least.  Truncating the states in the transformed system corresponding to small Hankel singular values again yields the reduced order model.  The transfer function error bound (\ref{eqn:BT_error_bound}) still holds, and the right hand side of the error bound is finite and tends to zero as $ r $ increases.

We can write our linear cable-mass system in the above first order abstract form (\ref{eqn:inf_dim_linear_sys}) with Hilbert space $ X = \mathcal{H} = V \times H $, as in Section \ref{subsec:linear_theory}.  The operator $ \mathcal{A} $ was defined previously.  The operators $ \mathcal{B}:\mathbb{R}\rightarrow \mathcal{H}$ and $ \mathcal{C} : \mathcal{H} \to \mathbb{R}^2 $ are defined as follows. First, $\mathcal{B}u=\left[0, B_{0}u\right]$, where $B_{0}u=\left[0,u,0\right]$.  Then let $x=[z, \chi] \in \mathcal{H}$, where $z$
is the position and $\chi$ is the velocity.  For $z=[w,w_{0},w_{l}]$ and $\chi=[p,\,p_{0},\,p_{l}]$, $ \mathcal{C} x = [ w_l, p_l ]^T $.  It can be checked that $ \mathcal{B} $ and $ \mathcal{C} $ are bounded, and therefore the balanced truncation theory holds for the linear cable-mass system.

We note that verifying the balanced truncation theory for PDE systems with inputs and/or outputs on the boundary of the spatial domain can often be very challenging \cite{curtain2001compactness, guiver2014model, Opmeer08} since the operators $ \mathcal{B} $ and/or $ \mathcal{C} $ are no longer bounded.  However, in our case, the input and output appearing in the boundaries with the dynamic boundary conditions cause the operators $ \mathcal{B} $ and $ \mathcal{C} $ to be bounded, and so we avoid the additional difficulty.

\subsection{Formulating the Finite Difference Approximation}

Finding exact solutions of the nonlinear cable-mass problem is
usually impossible. Therefore we use a basic numerical method, the finite
difference method, to approximate the solution to our model problem
with dynamic boundary conditions. Using this method we approximate our PDE system
by a large ODE system, and we apply the model reduction method to the resulting nonlinear finite dimensional system.

We place $n$ equally spaced nodes $\{ x_{j}\} _{j=1}^{n}$
in the interval $\left[0, l\right]$, where $x_{j}=(j-1)h$ and $h=l/(n-1)$
so that $x_{1}=0$ and $x_{n}=l$. In order to apply balanced truncation
below, we also eliminate the second order time derivatives by introducing a
velocity variable. Therefore, let $d_{i}$ denote the finite difference
approximation to the displacement $w(t,x_{i})$, and let $v_{i}$
denote the finite difference approximation to the velocity $w_{t}(t,x_{i})$.  We assume the solution is smooth so that the displacement
and velocity compatibility conditions are satisfied; we obtain
\begin{alignat*}{1}
w_{0}(t)=d_{1}(t), & \,\,\,\,\,\,w_{l}(t)=d_{n}(t),\\
\dot{w}_{0}(t)=v_{1}(t), & \,\,\,\,\,\,\dot{w}_{l}(t)=v_{n}(t).
\end{alignat*}
We use second order centered differences to form finite difference
equations for the wave equation (\ref{eq:(1a)})
\begin{align}
v_{i}' &= \frac{\gamma}{h^{2}}[v_{i+1}-2v_{i}+v_{i-1}]+\frac{\beta^{2}}{h^{2}}[d_{i+1}-2d_{i}+d_{i-1}]-\alpha v_{i},\nonumber \\
d_{i}' &= v_{i},  \quad  \mbox{for $i=2,\dots,n-1 $}.\label{eq:discritization terms}
\end{align}
To discretize our system we use (\ref{eq:discritization terms}) to
obtain
\begin{align*}
v_{i}' &= \left[-\alpha-\frac{2\gamma}{h^{2}}\right]v_{i}+\left[\frac{\gamma}{h^{2}}\right]v_{i-1}+\left[\frac{\gamma}{h^{2}}\right]v_{i+1}+\left[\frac{\beta^{2}}{h^{2}}\right]d_{i+1}-\left[\frac{2\beta^{2}}{h^{2}}\right]d_{i}+\left[\frac{\beta^{2}}{h^{2}}\right]d_{i-1},\\
d_{i}' &= v_{i},  \quad  \mbox{for $i=2,\dots,n-1 $}.
\end{align*}

To discritize the dynamic boundary conditions we use second order accurate
one-sided finite difference approximation to the first order spatial
derivatives, i.e.,
\begin{align*}
w_{x}(t,x)&\approx\frac{-3w(t,x)+4w(t,x+h)-w(t,x+2h)}{2h},\\
w_{x}(t,x)&\approx\frac{3w(t,x)-4w(t,x-h)+w(t,x-2h)}{2h}.
\end{align*}
Using these approximations we discretize $w_x(t,0)$ in left boundary
condition and $w_x(t,l)$ in right boundary
condition by
\begin{align*}
w_x(t,0) &\approx\frac{-3d_{1}+4d_{2}-d_{3}}{2h},\\
w_x(t,l) &\approx\frac{3d_{n}-4d_{n-1}+d_{n-2}}{2h}.
\end{align*}
Using these one-sided finite difference approximations allows us to keep the second order accuracy without introducing ``ghost'' nodes outside of the spatial domain.  After discretizing the dynamic boundary conditions we obtain
\begin{align*}
  v_{1}' &= \left[-\frac{k_{0}}{m_{0}}-\frac{3\beta^{2}}{2hm_{0}}\right]d_{1}+\left[\frac{4\beta^{2}}{2hm_{0}}\right]d_{2}-\left[\frac{\beta^{2}}{2hm_{0}}\right]d_{3}+\left[-\frac{3\gamma}{2hm_{0}}-\frac{\alpha_{0}}{m_{0}}\right]v_{1}\\
 & \quad +\left[\frac{4\gamma}{2hm_{0}}\right]v_{2}-\left[\frac{\gamma}{2hm_{0}}\right]v_{3}+\frac{u(t)}{m_{0}},\\
  d_{1}'&=v_{1},\\
  v_{n}' &= \left[-\frac{k_{l}}{m_{l}}-\frac{3\beta^{2}}{2hm_{l}}\right]d_{n}+\left[\frac{4\beta^{2}}{2hm_{l}}\right]d_{n-1}-\left[\frac{\beta^{2}}{2hm_{l}}\right]d_{n-2}\,\\
 & \quad +\left[-\frac{\alpha_{l}}{m_{l}}-\frac{3\gamma}{2hm_{l}}\right]v_{n}+\left[\frac{4\gamma}{2hm_{l}}\right]v_{n-1}-\left[\frac{\gamma}{2hm_{l}}\right]v_{n-2}-\left[\frac{k_{3}}{m_{l}}\right]\left[d_{n}\right]^{3},\\
  d_{n}'&=v_{n}.
\end{align*}

Then the matrix form of the above system becomes
\begin{align*}
\left[\begin{array}{c}
d'\\
v'
\end{array}\right] &= \left[\begin{array}{cc}
0 & I\\
A_{11} & A_{12}
\end{array}\right]\left[\begin{array}{c}
d\\
v
\end{array}\right]+
\left[\begin{array}{c}
0\\
F_0(d)
\end{array}\right]+\left[\begin{array}{c}
0\\
B_{1}
\end{array}\right]u,\\
\left[\begin{array}{c}
y_{1}\\
y_{2}
\end{array}\right] &= \left[\begin{array}{ccccc}
0 & \cdots & 1 & \cdots & 0\\
0 & \cdots & 0 & \cdots & 1
\end{array}\right]\left[\begin{array}{c}
d\\
v
\end{array}\right],
\end{align*}
where
$$
  F_0(d) = [0, \ldots, 0, -m_l^{-1} k_3 d_n^3 ]^T.
$$
Or, we can write the nonlinear finite dimensional approximating system as
\begin{equation}\label{eqn:nonlinear_FOM}
  \dot{x} = A x + F(x) + B u,  \quad  y = C x.
\end{equation}
%

\subsection{Implementation of the Balanced Truncation Method}

We compute the balanced truncated reduced order model using the ``square
root algorithm'' described in \cite{Antoulas05}. The algorithm generate matrices $T_{r} \in R^{2n\times r}$
and $S_{r}  \in R^{r\times2n}$ such that $T_{r}=\left[\varphi_{1},\varphi_2,\ldots,\varphi_{r}\right]$,
where $\varphi_{j}$ denotes the $j$th column of $T_{r}$, and $S_{r}=\left[\psi_{1},\psi_2,\ldots,\psi_{r}\right]^{T}$,
where $\psi_{i}$ denotes the $i$th row of $S_{r}$. Also, $S_{r}T_{r}=I_{r}$, where $I_{r}$ is the identity matrix.

Approximate $ x(t) $ in the nonlinear full order model (\ref{eqn:nonlinear_FOM}) by $ x(t) \approx T_r a(t) $, and multiply the full order model on the left by $ S_r $ to produce the reduced order model
\[
  \dot{a}(t)=A_{r}a(t)+B_{r}u(t)+S_{r}\,F(T_{r}a),  \quad  y_{r}(t)=C_{r}a(t).
\]
The matrices $ A_r $, $ B_r $, and $ C_r $ in the reduced order model are given by $A_{r}=S_{r}\,A\,T_{r}$, $B_{r}=S_{r}\,B$, and $C_{r}=C\,T_{r}$.  These are exactly the same matrices from the balanced truncated reduced order model in the linear case.

We want to rewrite the nonlinear term $S_{r}\,F(T_{r}a)$ so that we can compute the reduced
order model using only low order operations. The $2n\times r$
matrix $T_{r}$ has $ ij $ entries $\varphi_{j,i}$, where $\varphi_{j,i}$ denotes the $ i $th entry of $ \varphi_j $, for $i=1,\ldots,2n$ and $j=1,\ldots,r$.
Also, $a=\left[a_{1},a_{2},\ldots,a_{r}\right]^{T}$ is a $r\times1$~dimension
vector. Then $ \sum_{j=1}^{r}\varphi_{j,i}a_{j}$ is
the $i$th entry of the vector $T_{r}a$, and so
\begin{equation}\label{eqn:Tra_multiply}
  T_{r}a=\left[ \sum_{j=1}^{r}\varphi_{j,1}a_{j},\ldots,\sum_{j=1}^{r}\varphi_{j,2n}a_{j} \right]^{T}.
\end{equation}
In our system we have only one nonlinear term that is in the right
boundary condition.  Let $ d_r $ be the vector consisting of the first $ n $ entries of $ T_r a $.  Using the definition of $ F(x) $ from the previous section gives
$$
  F(T_r a) = \left[\begin{array}{c}
             0\\
             F_0(d_r)
             \end{array}\right],  \quad  F_0(d_r) = [0, \ldots, 0, -m_l^{-1} k_3 ( T_r a )_n^3 ]^T,
$$
where $ ( T_r a )_n $ denotes the $ n $th entry of $ T_r a $.  Therefore, we do not need to compute the entire $ 2n \times 1 $ vector $ T_r a $ as in (\ref{eqn:Tra_multiply}); we only need the $ n $th entry.  This gives
$$
\left[F(T_{r}a)\right]_{j}=\begin{cases}
0, & j \neq 2n,\\
-m_l^{-1} k_3 \left( \sum_{j=1}^{r}\varphi_{j,n} a_{j} \right)^{3},  &  j = 2n.
\end{cases}
$$
Therefore, the nonlinear term in the reduced order model can be computed using only low order operations by
\[
  [ S_{r}\,F(T_{r}a) ]_i = -m_l^{-1} k_3 \psi_{i,2n} \, \bigg(\sum_{j=1}^{r} \varphi_{j,n} a_{j} \bigg)^{3},  \quad  i=1,2,\ldots,r,
\]
where $ \psi_{i,2n} $ denotes the $ 2n $th entry of the vector $ \psi_i $.

\section{Numerical Results}
\label{sec:numerical_results}

In this section, we present numerical results concerning the effectiveness of the balanced truncation MOR method applied to the finite difference approximation of the cable-mass problem.  For our experiments, we used 100 finite difference nodes and solved all ordinary differential equations with Matlab's \texttt{ode23s}.  Increasing the number of nodes did not change the results.  We fixed selected system parameters, as shown in Table \ref{table 1}, and tested variations of the remaining parameters to determine when the MOR approach is accurate.
\begin{table}[htb]
\renewcommand{\arraystretch}{1.5}
\centering
\caption{Fixed simulation parameters}
{\begin{tabular}{ccccc}
  $ l $ & $ m_0 $ & $ m_l $ & $ k_3 $ & $ \beta $ \\
  \hline
  1 & 1 & 1.5 & 1 & 1 \\
\end{tabular}}
\label{table 1}
\end{table}

We investigated the following examples:
\begin{description}
\item[Example 1] Kelvin-Voigt damping in the interior $(\gamma>0)$ and
damping in the in the right boundary $(\alpha_{l}>0)$. All other
damping damping parameters are taken to be zero, i.e., $\alpha_{0}=\alpha=0$.
\item[Example 2] Viscous damping in the interior $(\alpha>0)$ and damping
in both boundaries $(\alpha_{0}, \alpha_{l}>0)$. The Kelvin-Voigt
damping parameter $\gamma$ is set to zero. Unlike Examples 1 and 3, the
$C_{0}$-semigroup generated by the linear problem is not analytic
in this case and the PDE is hyperbolic.
\item[Example 3] Viscous damping in the interior $(\alpha>0)$ and Kelvin-Voigt
damping in the interior $(\gamma>0)$. All other damping damping
parameters are taken to be zero, i.e., $\alpha_{0}=\alpha_{l}=0$.
\item[Example 4] Viscous damping in the interior $(\alpha>0)$. All other
damping damping parameters are taken to be zero, i.e., $\gamma=\alpha_{0}=\alpha_{l}=0$.
\item[Example 5] Kelvin-Voigt damping in the interior $(\gamma>0)$. All
other damping damping parameters are taken to be zero, i.e., $\alpha=\alpha_{0}=\alpha_{l}=0$.
\end{description}
In Sections \ref{subsec:linear_theory} and \ref{sec:nonlinear_problem} we proved that the unforced linear and nonlinear systems are exponentially stable for Examples 1--3.  Numerical results (not shown) indicate that the linear problems are also exponentially stable for Examples 4--5.  It can be checked that the damping bilinear form $\sigma_{2}$ is not $H$-elliptic for these last two cases.  Therefore, the exponential stability results in Sections \ref{subsec:linear_theory} and \ref{sec:nonlinear_problem} do not apply to Examples 4--5.  We leave theoretical analysis of these cases to be considered elsewhere;  however, we do investigate the model reduction performance computationally.
%

\subsection{Exponential Stability}

Before we present the model reduction computational results, we briefly present numerical results concerning the linear and nonlinear exponential stability theory.  For the linear problem, we test the exponential stability by analyzing the eigenvalues of the matrix $ A $ in the finite difference model (\ref{eqn:nonlinear_FOM}).  Figure \ref{fig:exp_stab_Ex1_linear} shows the eigenvalues of $ A $ for $\gamma=\alpha_{l}=0.1 $, $k_{0}=k_{l}=1$, and $\alpha_{0}=\alpha=0$ (this is a case of Example 1), and they all have negative real part.  For the nonlinear problem, we consider the solution of the finite difference model (\ref{eqn:nonlinear_FOM}), and compute an approximation to the (continuous) energy function in Theorem \ref{thm:exponential stability of nonlinear system} by using trapezoid rule quadrature on the integrals.   Figure \ref{fig:exp_stab_Ex1_nonlinear} shows the exponential decay of the energy with the same parameters and the initial data $e^{x}\sin(1-x)$ for the position and $ \cos(x) $ for the velocity.
\begin{figure}[htb]
\centering
\subfloat[Linear system]{\includegraphics[width=5cm,height=5cm]{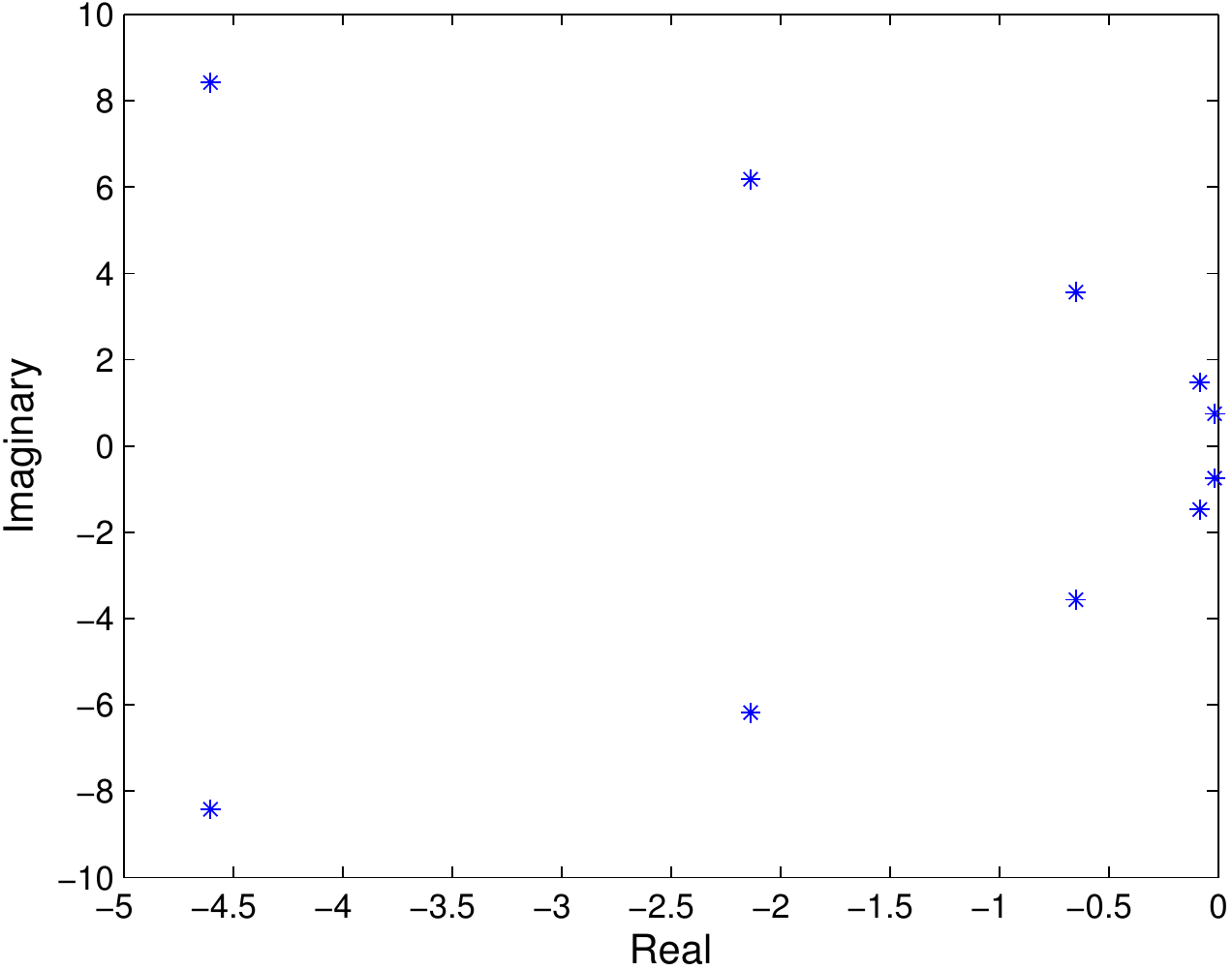}\label{fig:exp_stab_Ex1_linear}}
\subfloat[Nonlinear system]{\includegraphics[width=5cm,height=5cm]{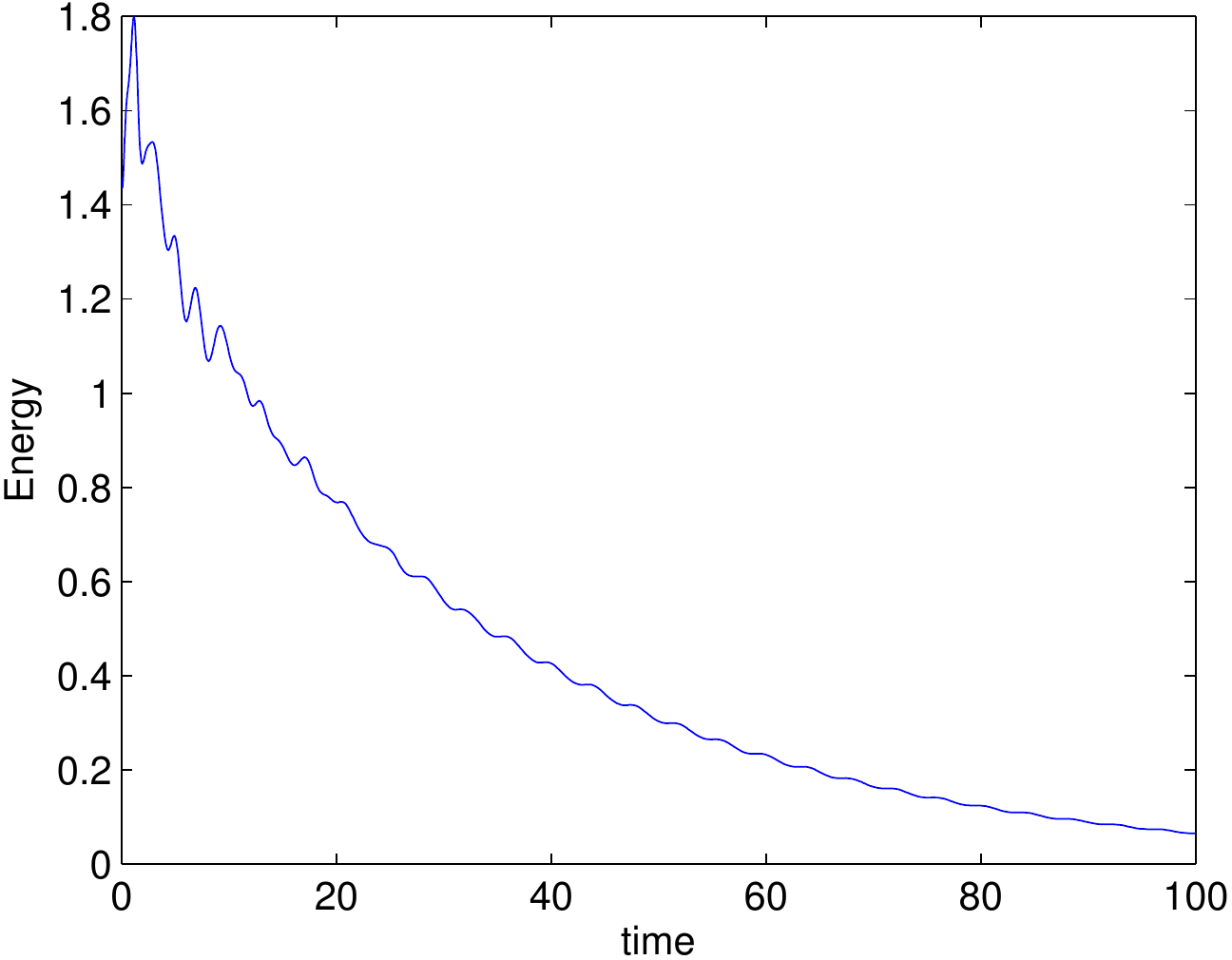}\label{fig:exp_stab_Ex1_nonlinear}}
\caption{Eigenvalues of the linear system and energy decay for the nonlinear system with $\gamma=\alpha_{l}=0.1 $, $k_{0}=k_{l}=1$, and $\alpha_{0}=\alpha=0$}
\end{figure}

We also approximated the eigenvalues and the energy function for the nonlinear problem when $\gamma=0$ (this is a case of Example 2); see Figure \ref{fig:exp_stability2}.  We see the exponential stability in both the linear and nonlinear cases.  The $ C_0 $-semigroup is not analytic in this case, and we see the imaginary part of the eigenvalues increase as is usual with hyperbolic problems.  In the nonlinear case, if $\gamma=0$ and all the other parameters are small (as in the figure) then the energy decays exponentially but also fluctuates rapidly.
\begin{figure}[htb]
\centering
\subfloat[Linear system]{\includegraphics[width=5cm,height=5cm]{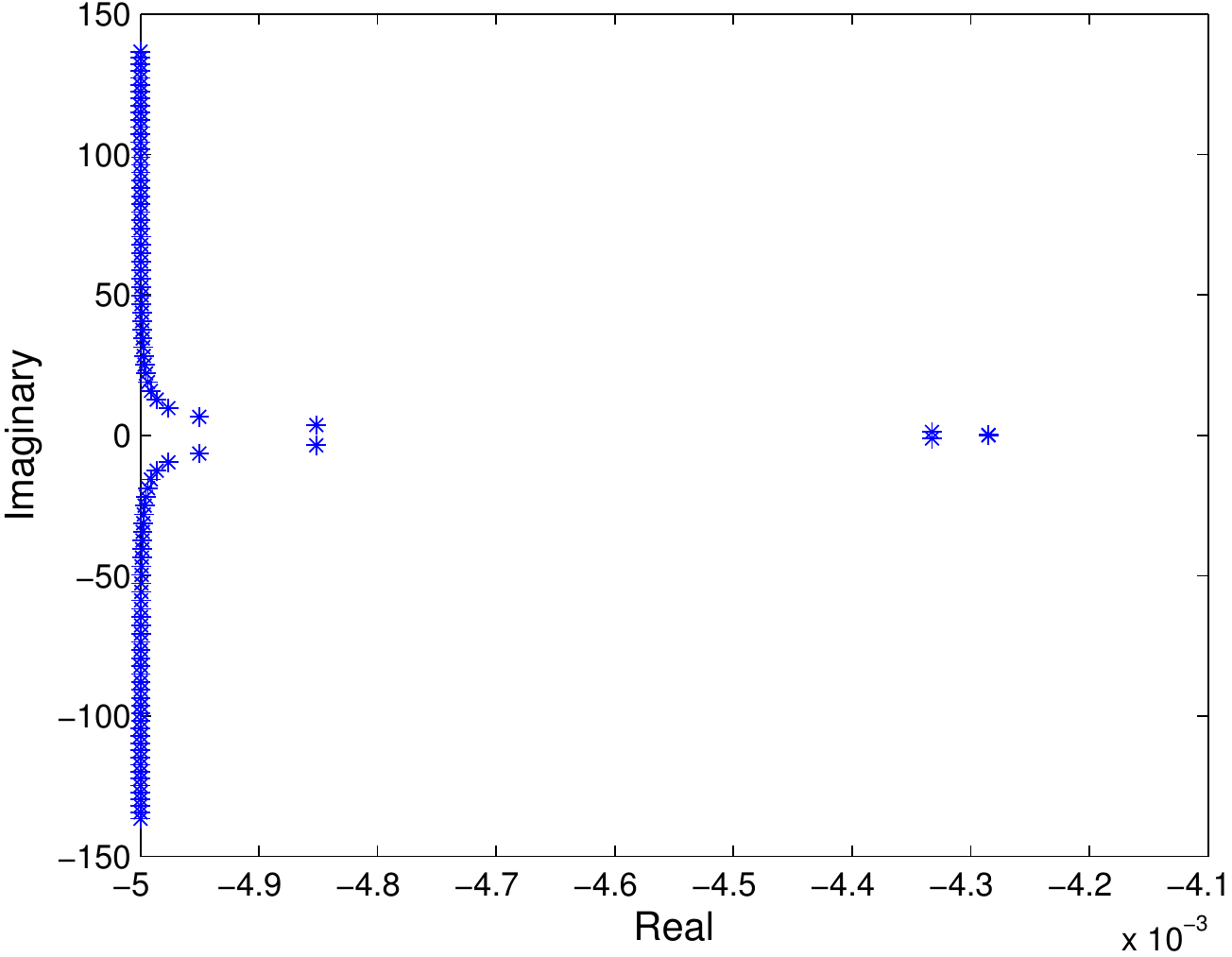}}
\subfloat[Nonlinear system]{\includegraphics[width=5cm,height=5cm]{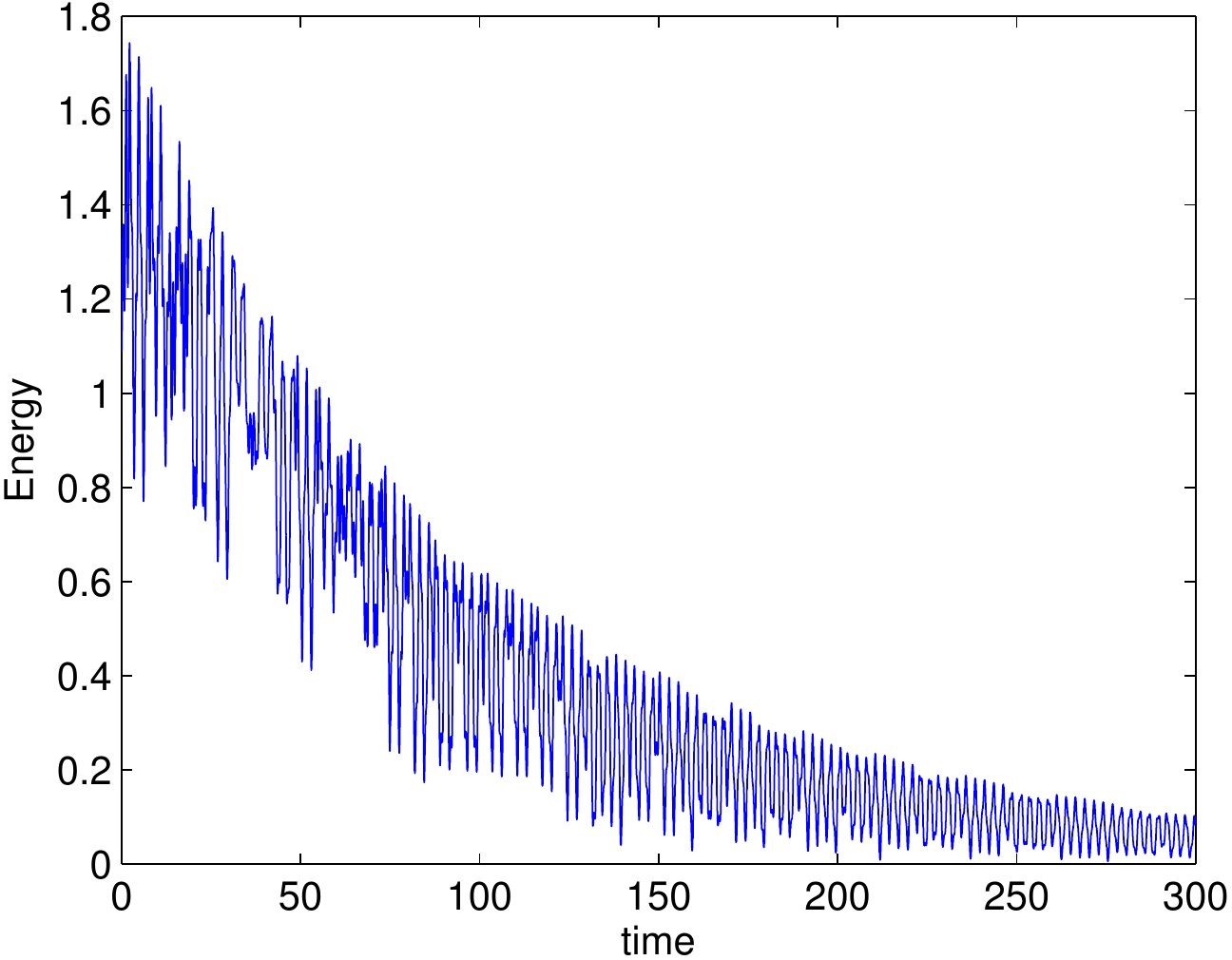}}
\caption{\label{fig:exp_stability2}Eigenvalues of the linear system and energy decay for the nonlinear system with $\gamma=0$ and $\alpha=\alpha_{0}=\alpha_{l}=k_{0}=k_{l}=0.01$}
\end{figure}

\subsection{Model Reduction Results}

Next, we begin the model reduction experiments.  We study the effects of
the various parameters on the accuracy of the model reduction. To
do this, we consider the nonlinear reduced order model (ROM) and full order
model (FOM) with zero initial data and the same input $u(t)$ and
compare the output of the FOM and ROM.  Recall the output $y(t)$ of
the cable-mass system is the position and velocity of the right mass.

Although we focus on the accuracy of the nonlinear ROM, we also present some results for the linear ROM for comparison.  The output of the linear ROM is highly accurate in all cases considered, as expected by balanced truncation theory.

For our experiments, we consider four different oscillating input functions $ u(t) $:
\begin{description}
\item[Input 1] $ u(t) = 0.1 \sin(0.2\pi t) $
\item[Input 2] $u(t)=0.02\cos(at)+0.03\cos(bt)$, where $a$, $b$ are the two largest real parts of the eigenvalues of the matrix $ A $
\item[Input 3] $u(t)=c_{1}\sin(mt)+c_{2}\cos(nt)$, for various constants $ c_1 $, $ c_2 $, $ m $, and $ n $
\item[Input 4] $ u(t) = 0.1 \, \mathrm{square}(0.2\pi t) $
\end{description}
Input 1 was originally considered for this problem in \cite{Shoori14}, where this cable-mass system was considered as a heuristic model for a wave tank with a wave energy converter.  Input 2 was considered for a different cable-mass problem in \cite{burns1998reduced}.  We note that this input causes a type of resonance, i.e., the solution magnitude can initially grow in time before the damping causes the magnitude to return to a moderate level.  We also considered Input 3 to test a variety of oscillating input behaviors.  Finally, we considered Input 4 to see if a discontinuous input\footnote{The square wave is defined by $ \mathrm{square}(t) = 1 $ if $ \sin(t) > 0 $ and $ \mathrm{square}(t) = -1 $ if $ \sin(t) < 0 $.} causes any change in the ROM output.

\textbf{Case: Small damping parameters.}  We first investigate the behavior of the ROM for damping parameters that are small relative to the boundary stiffness parameters.  In this case, for all examples and inputs, the output of the nonlinear ROM is highly accurate compared to the FOM output.  We present results for two specific scenarios.  Figure \ref{fig:small_damp_ex1_in2} shows the output of the FOM and ROM for both the linear and nonlinear systems for Example 1, Input 2 with $\alpha_{0}=\alpha=0$, $\alpha_{l}=k_{0}=k_{l}=0.1$, and small Kelvin-Voigt parameter $\gamma=0.001$.  The agreement is excellent in both the linear and nonlinear cases.
\begin{figure}[htb]
\centering{}\subfloat[Linear system]{\includegraphics[width=6cm,height=6cm]{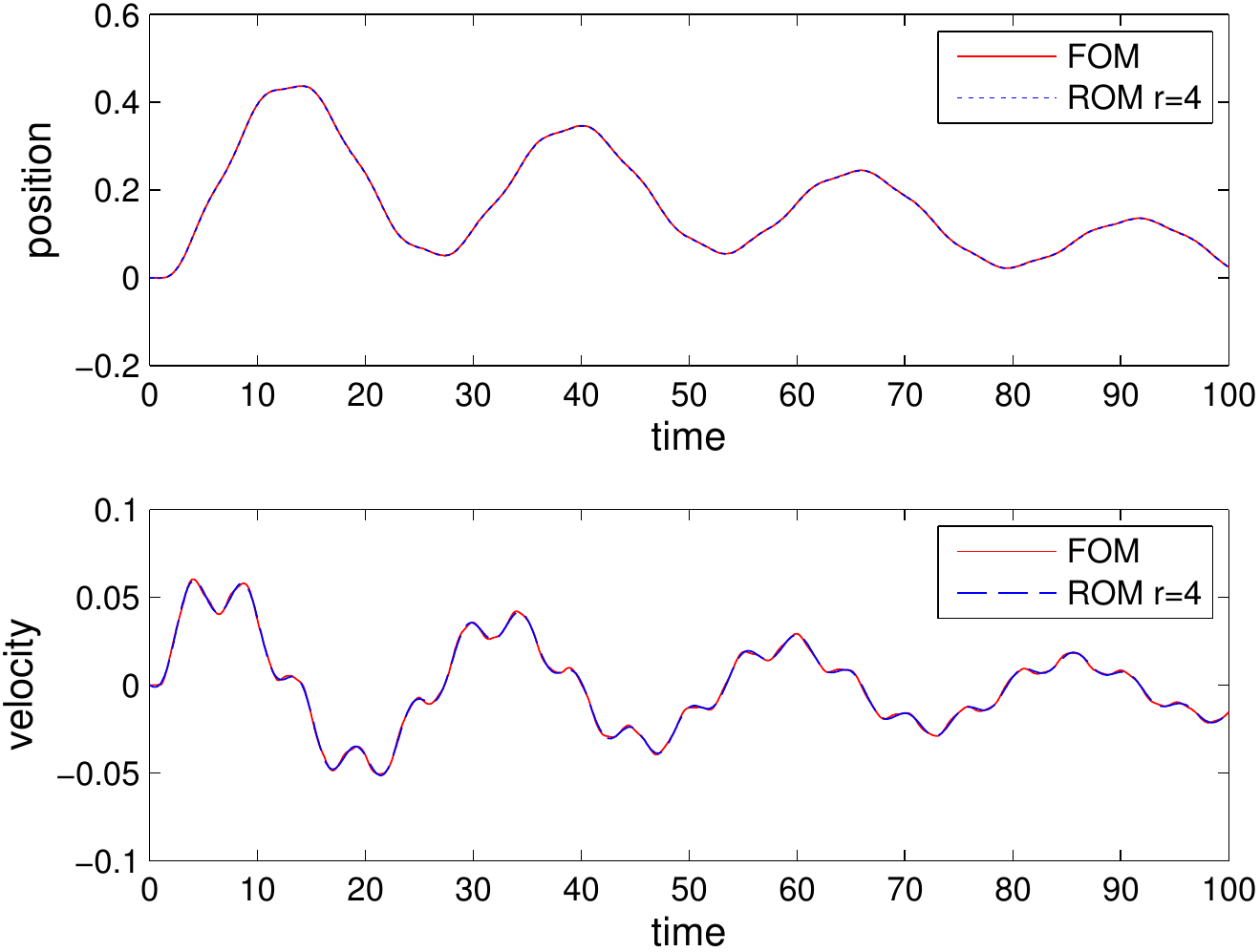}}
\subfloat[Nonlinear system]{\includegraphics[width=6cm,height=6cm]{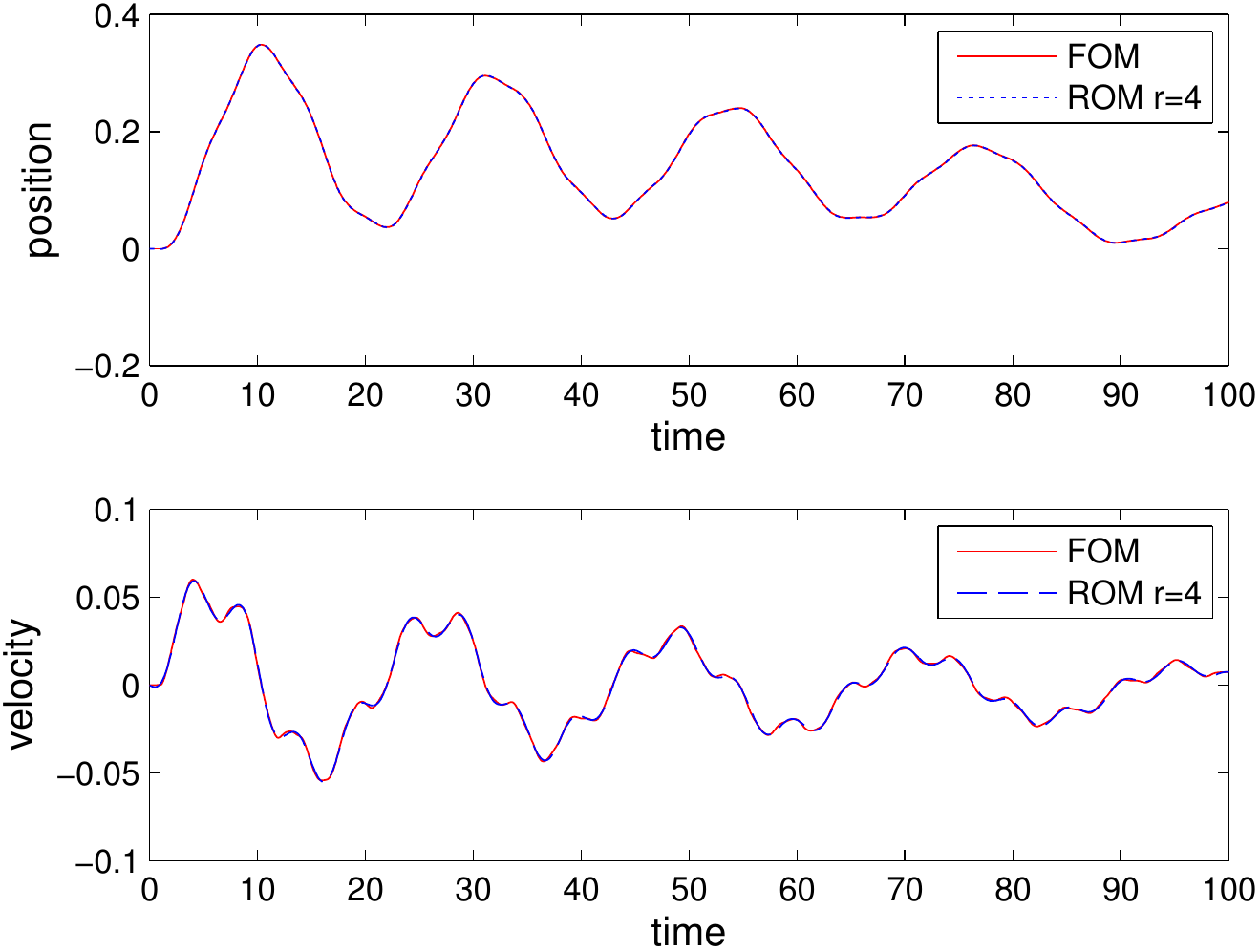}}
\caption{\label{fig:small_damp_ex1_in2}Example 1, Input 2: Output of the ROM and FOM for $\alpha_{0}=\alpha=0$, $\alpha_{l}=k_{0}=k_{l}=0.1$, and $\gamma=0.001$}
\end{figure}

Next we observe the behavior of the ROM and FOM for small damping with discontinuous input. Figure \ref{fig:small_damp_ex5_in4} shows the behavior of the nonlinear FOM and ROM for Example 5, Input 4 (the square wave) with $\alpha=\alpha_{0}=\alpha_{l}=0$, $\gamma=0.001$, and $k_{0}=k_{l}=0.1$.  The nonlinear ROM is highly accurate even though the position and velocity outputs are very irregular.
\begin{figure}[htb]
\begin{centering}
\includegraphics[width=6cm,height=6cm]{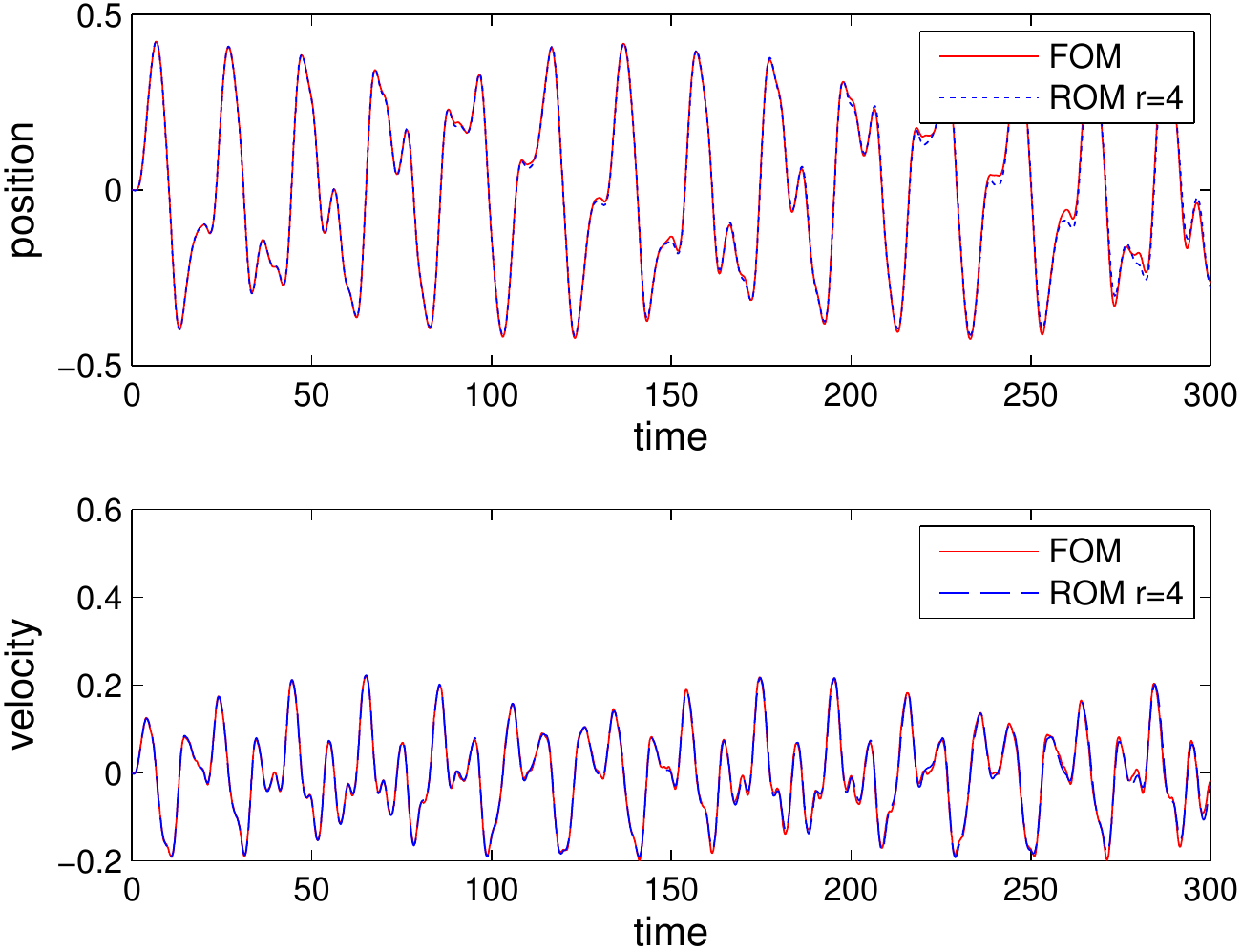}
\caption{\label{fig:small_damp_ex5_in4}Example 5, Input 4: Output of the nonlinear ROM and FOM for $\alpha=\alpha_{0}=\alpha_{l}=0$, $\gamma=0.001$, and $k_{0}=k_{l}=0.1$}
\end{centering}
\end{figure}

Again, we note that the nonlinear ROM is very accurate for all examples and all inputs when the damping parameters are small relative to the boundary stiffness parameters.

\textbf{Case: Small stiffness parameters.}  Next we investigate the behavior of the FOM and ROM when the boundary stiffness parameters are small relative to the damping parameters.  In this case, we notice a difference in the accuracy of the ROM depending on the smoothness of the input.

We being with smooth inputs, i.e., Input 1--3.  Overall, for all examples and all smooth inputs, the nonlinear ROM is highly accurate.  One specific scenario is shown in Figure \ref{fig:small_stiff_ex2_in1}.  Here, the linear and nonlinear FOM and ROM output are shown for Example 2, Input 1 with $\gamma=0$, $\alpha=\alpha_{0}=\alpha_{l}=0.1$, and small stiffness parameters $k_{0}=k_{l}=0.001$.  We note that the linear and nonlinear ROM output are very accurate.
\begin{figure}[htb]
\centering{}\subfloat[Linear system]{\includegraphics[width=6cm,height=6cm]{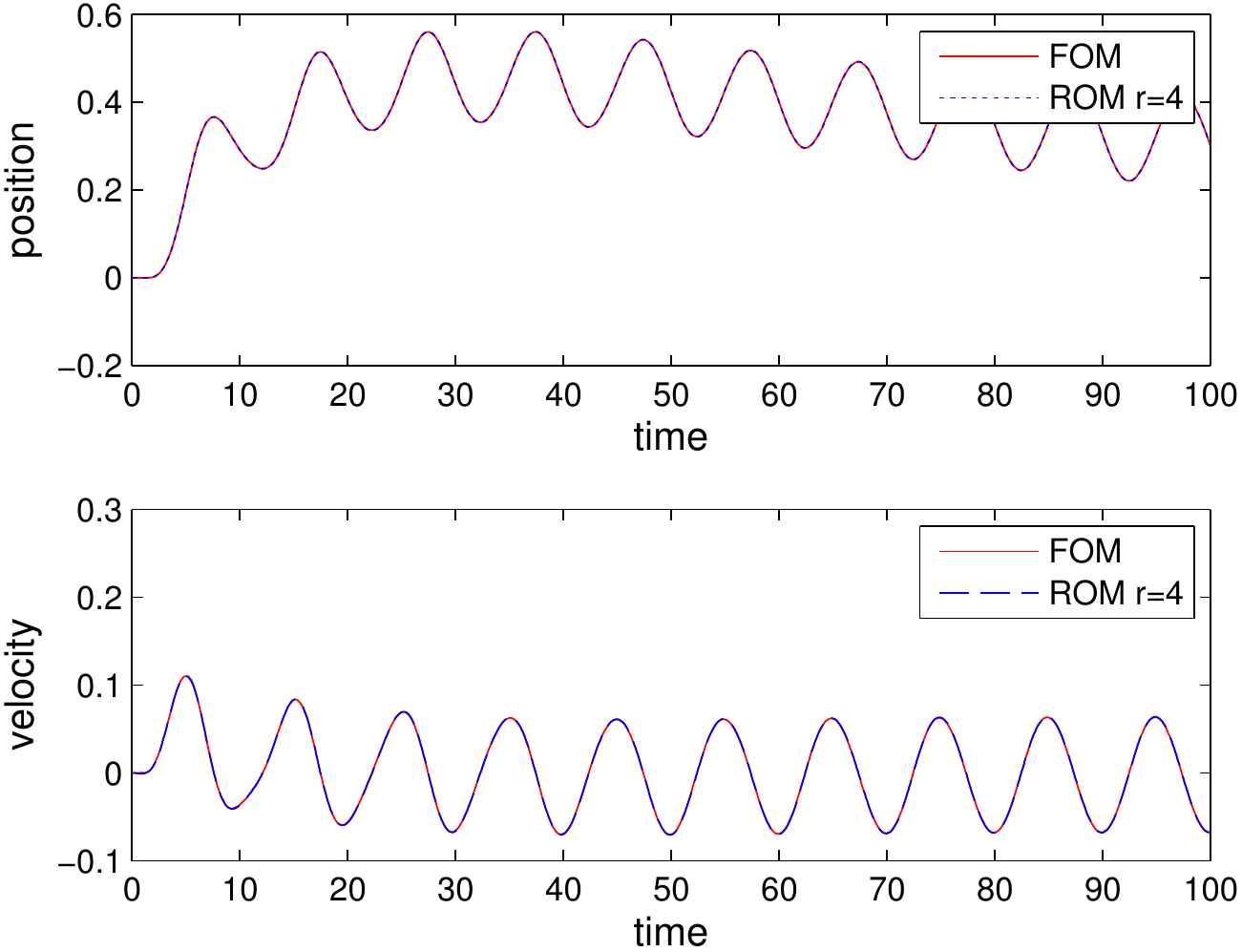}}
\subfloat[Nonlinear system]{\includegraphics[width=6cm,height=6cm]{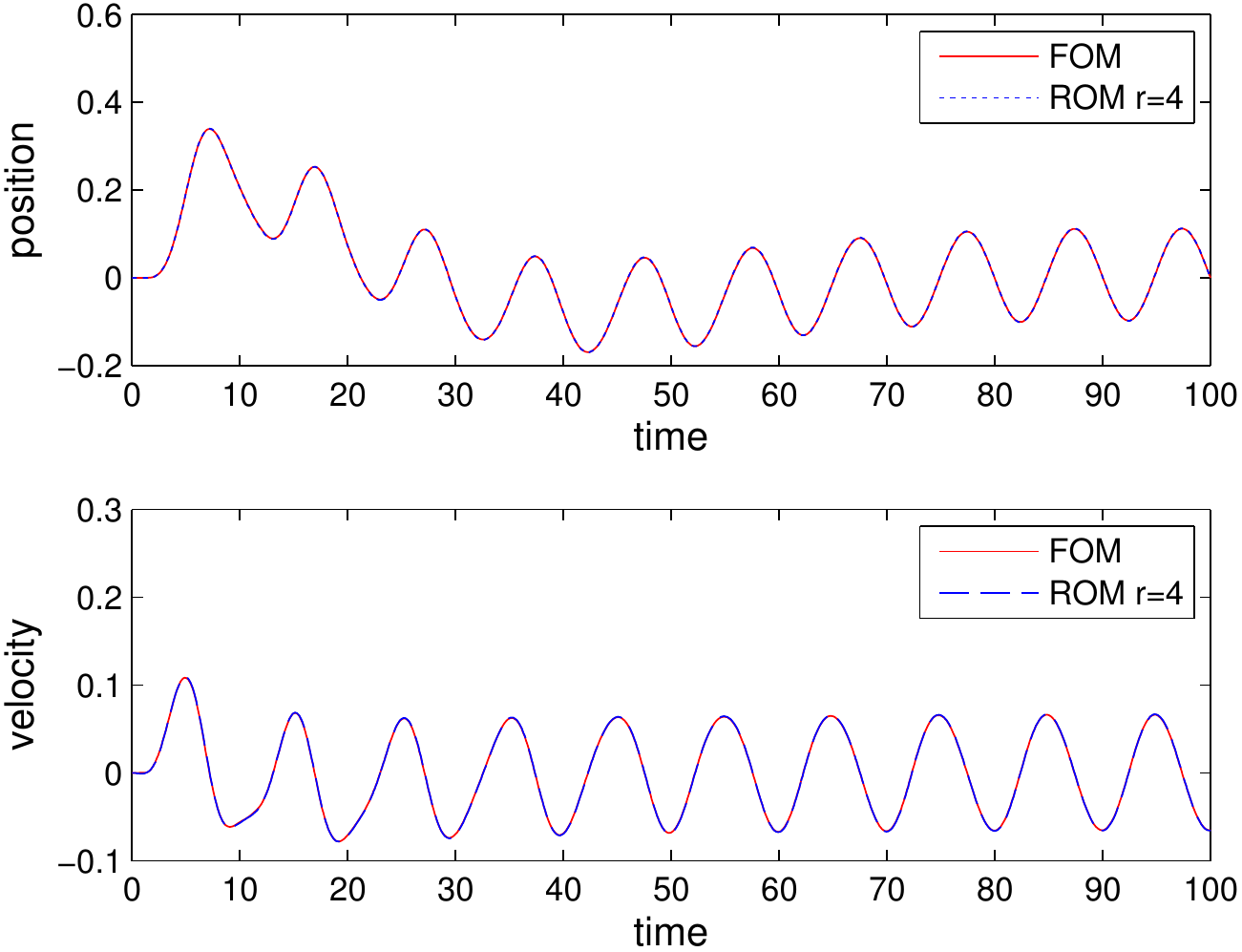}}
\caption{\label{fig:small_stiff_ex2_in1}Example 2, Input 1: Output of the ROM and FOM for $\gamma=0$, $\alpha=\alpha_{0}=\alpha_{l}=0.1$, and $k_{0}=k_{l}=0.001$}
\end{figure}

The overall results change with Input 4, the discontinuous square wave.  We present results for two specific scenarios.  First, Figure \ref{fig:small_stiff_ex1_in4} shows the linear and nonlinear FOM and ROM output for Example 1 with $ \alpha = \alpha_0 = 0 $, $\gamma=\alpha_{l}=0.1$, and small stiffness $k_{0}=k_{l}=0.001$.  The nonlinear ROM for $ r = 4 $ is accurate over an initial time interval, but then shows a slight loss in accuracy. However, increasing $ r $ in the nonlinear ROM does yield high accuracy, even over a long time interval (not shown).  Figure \ref{fig:small_stiff_ex5_in4} shows the FOM and ROM output over the longer time interval $ 0 \leq t \leq 300 $ for another scenario: Example 5, Input 4 with $\gamma=0.1$, $\alpha=\alpha_{0}=\alpha_{l}=0$, and small stiffness $k_{0}=k_{l}=0.001$.  The nonlinear ROM output is again highly accurate for an initial time period, but then suffers a loss of accuracy.  Increasing $ r $ does not improve the accuracy.
\begin{figure}[htb]
\centering{}\subfloat[Linear system]{\includegraphics[width=6cm,height=6cm]{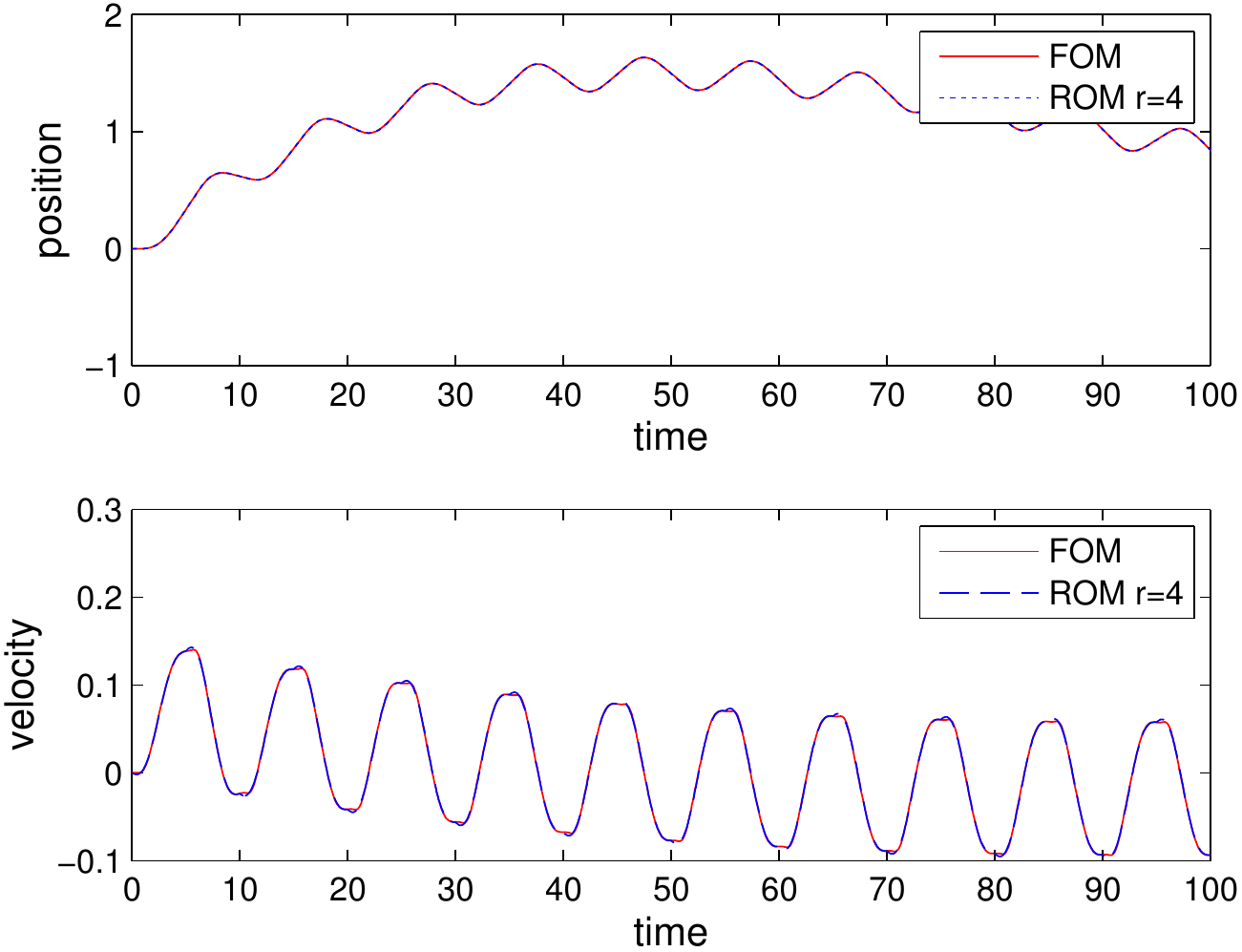}}
\subfloat[Nonlinear system]{\includegraphics[width=6cm,height=6cm]{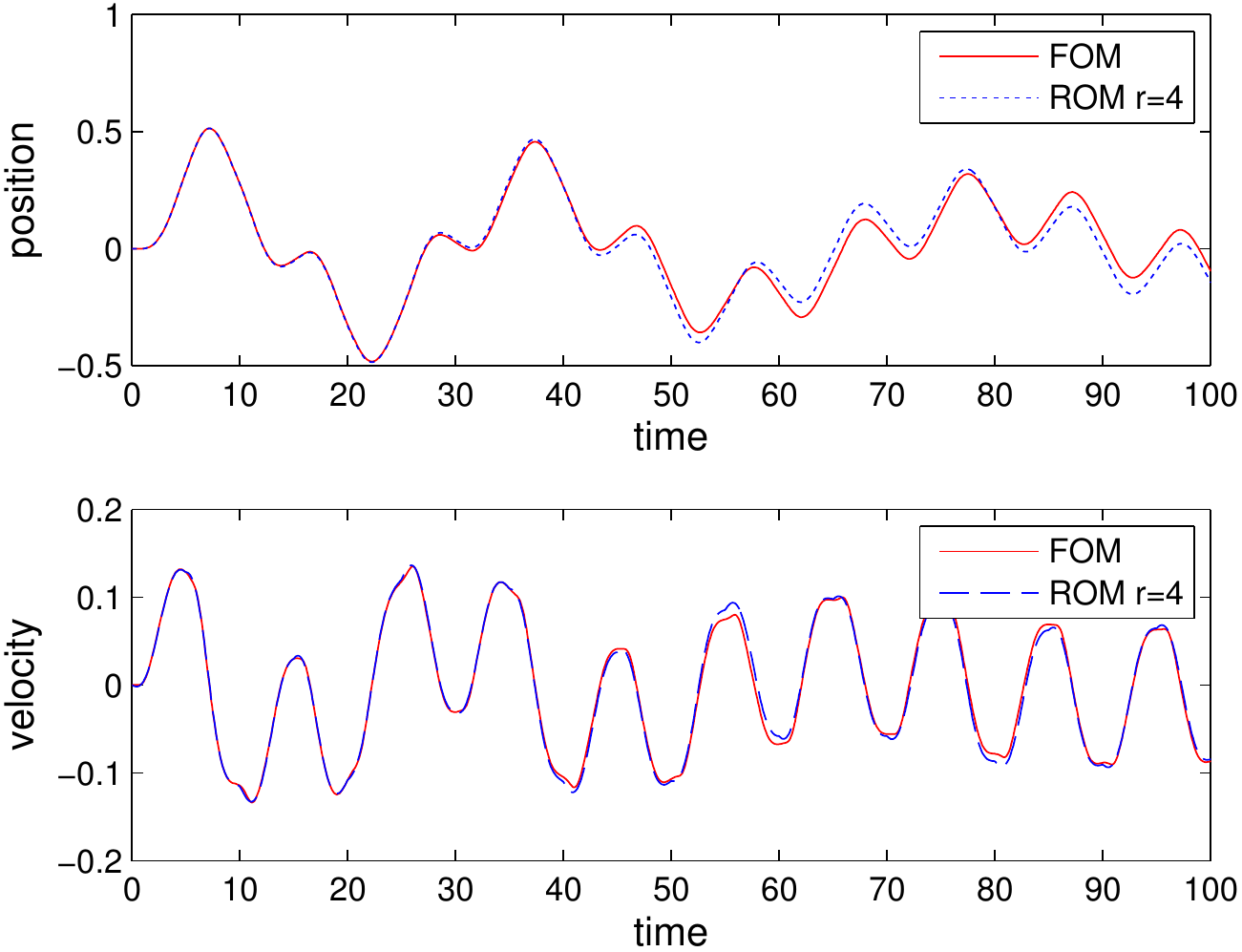}}
\caption{\label{fig:small_stiff_ex1_in4}Example 1, Input 4: Output of the ROM and FOM for $ \alpha = \alpha_0 = 0 $, $\gamma=\alpha_{l}=0.1$, and $k_{0}=k_{l}=0.001$}
\end{figure}
\begin{figure}[htb]
\centering{}\includegraphics[width=6cm,height=6cm]{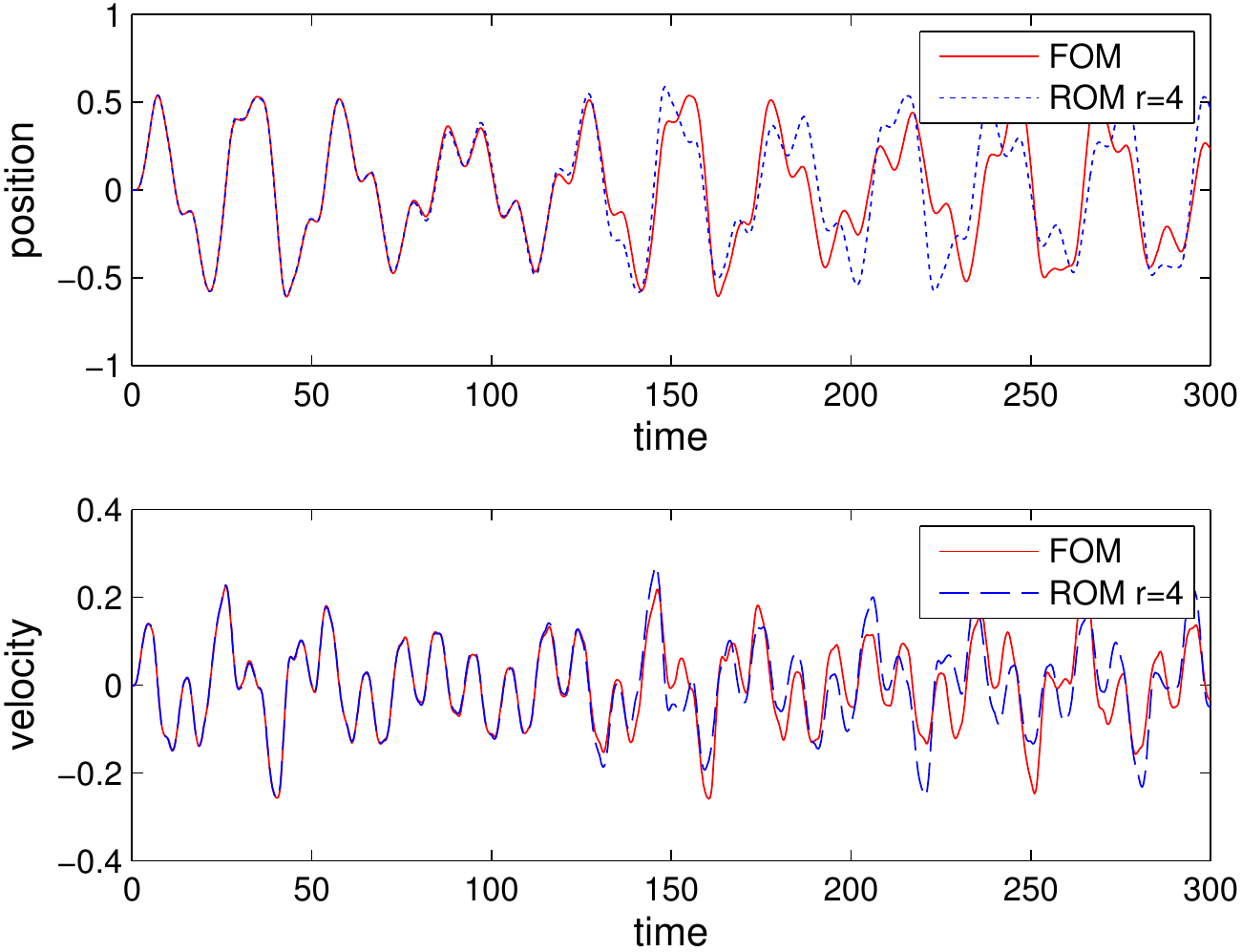}
\caption{\label{fig:small_stiff_ex5_in4}Example 5, Input 4: Output of the ROM and FOM for $\gamma=0.1$, $\alpha=\alpha_{0}=\alpha_{l}=0$, and $k_{0}=k_{l}=0.001$}
\end{figure}

To summarize, when the stiffness parameters are small relative to the damping parameters, the nonlinear ROM is highly accurate for all examples with smooth inputs.  However, for the discontinuous square wave input, the nonlinear ROM is only highly accurate over an initial time period and then accuracy can be lost.  Increasing the order $ r $ of the ROM may not improve the accuracy.  Also, if the magnitude of the input is reduced, then the length of the highly accurate initial time interval does increase (not shown).

\textbf{Case: Small damping and stiffness parameters.}  Finally, we consider the behavior of the nonlinear ROM when the damping and stiffness parameters are small relative to the mass and nonlinear stiffness parameters ($ m_0 = 1 $, $ m_l = 1.5 $, and $ k_3 = 1 $).  In this case, all behaviors are possible: the nonlinear ROM can be highly accurate over a long time interval, or it can lose high accuracy after an initial time period.  A loss of accuracy can occur for any example and for smooth or discontinuous inputs.  We present two specific scenarios.  Figure \ref{fig:small_all_ex3_in2_in4} shows the output of the nonlinear FOM and ROM for Example 3 with $\alpha_{0}=\alpha_{l}=0$, $\alpha=\gamma=k_{0}=k_{l}=0.001$, and two different inputs.  We see high accuracy over a long time interval for Input 2, but a loss of accuracy over a long time interval for Input 4.
\begin{figure}[htb]
\centering{}
\subfloat[Input 2]{\includegraphics[width=6cm,height=6cm]{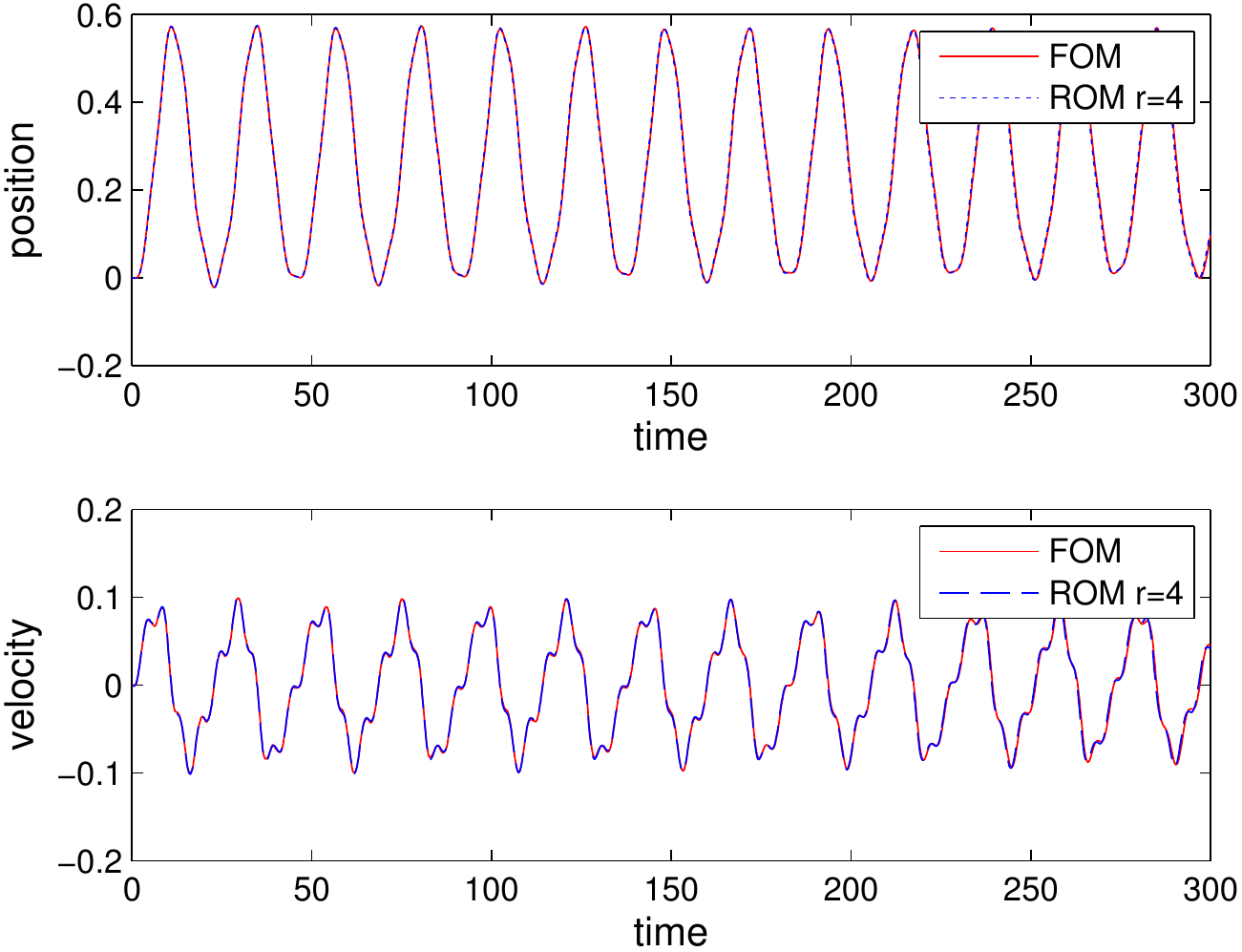}}
\subfloat[Input 4]{\includegraphics[width=6cm,height=6cm]{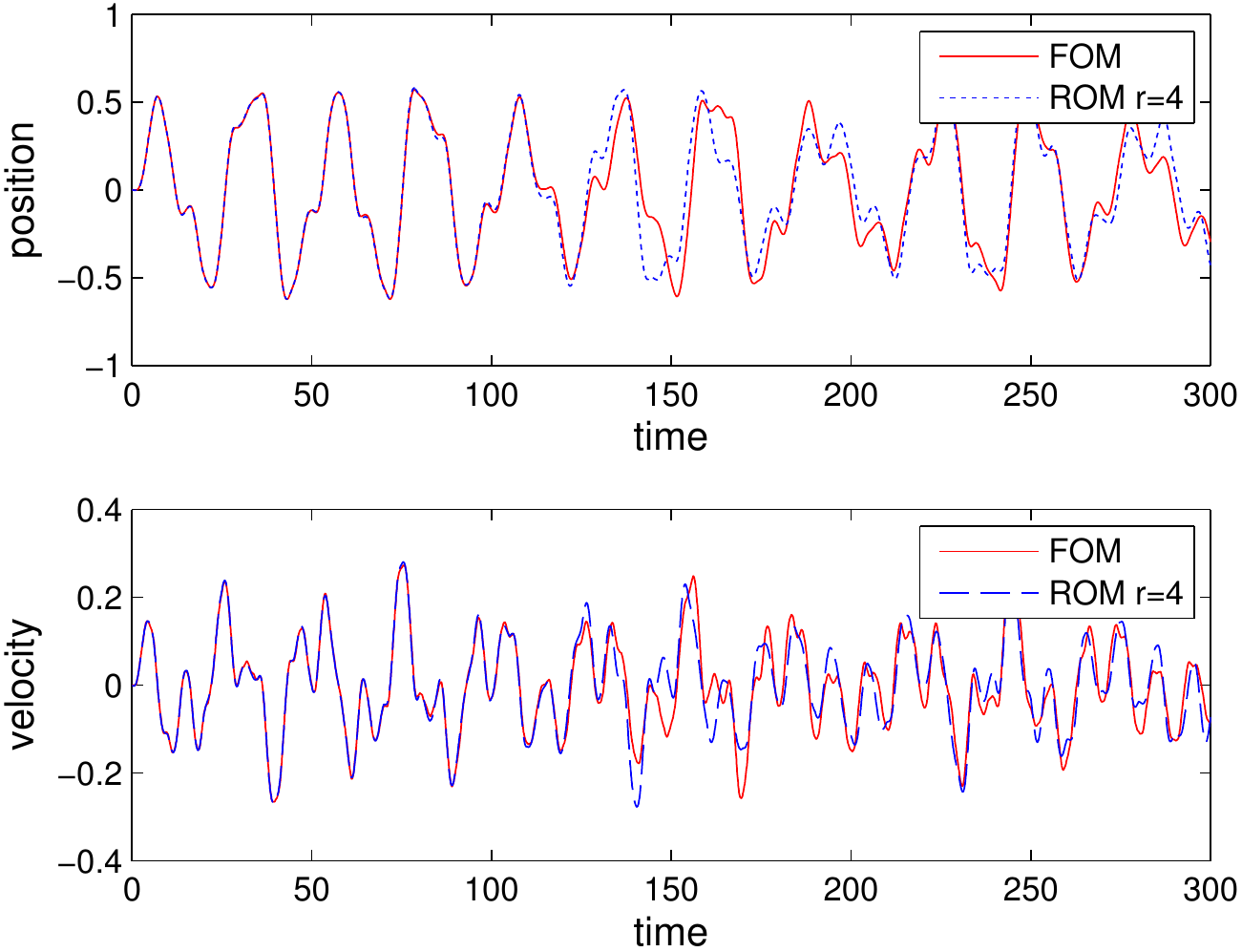}}
\caption{\label{fig:small_all_ex3_in2_in4}Example 3: Output of the nonlinear ROM and FOM for $\alpha_{0}=\alpha_{l}=0$ and $\alpha=\gamma=k_{0}=k_{l}=0.001$}
\end{figure}

We note that we have primarily focused on presenting results for scenarios where some of the system parameters are very small relative to other parameters.  We also tested many other parameters scenarios, and the nonlinear ROM was frequently highly accurate.

\section{Conclusion}

We considered a cable-mass system originally motivated by an application to wave energy that is modeled by a 1D wave equation with linear and nonlinear second order oscillator dynamic boundary conditions.  We proved the well-posedness of the unforced linear and nonlinear problems.  For certain assumptions on the damping parameters, we proved the linear problem is exponentially stable and the energy decays exponentially fast for the nonlinear problem.  This model primarily differs from most models considered in the literature because the dynamic boundary conditions hold on all boundaries.

For the forced input-output nonlinear cable-mass system, we described and numerically investigated a model order reduction (MOR) approach based on balanced truncation.  We found that the nonlinear reduced order model (ROM) was highly accurate for many parameter scenarios, including cases when the PDE model is parabolic or hyperbolic with light damping.  Even for the most challenging parameter cases, the nonlinear ROM was highly accurate for an initial time interval and the accuracy of the ROM increased when the magnitude of the input decreased.

No theory currently exists for this nonlinear MOR approach.  The results in this paper suggest some beginning theoretical results that could be investigated.  Specifically, our numerical results suggest it may be possible to prove that the error in the output for the nonlinear ROM is small over an initial time interval, and that the length of this interval increases as the magnitude of the input decreases.  Given the lack of theoretical results for many types of nonlinear model reduction, even a basic theoretical result such as this would be an advance.

For our linear and nonlinear exponential stability results, we required the damping bilinear form to be $ H $-elliptic.  We gave three examples of models with interior damping where this condition is satisfied.  We numerically considered two other parameter scenarios that appear to give exponential stability: the Kelvin-Voigt interior damping parameter $ \gamma $ is positive only, and the viscous interior damping parameter $ \alpha $ is positive only.  In these two cases the damping bilinear form is not $ H $-elliptic, and so the exponential stability theory in our work does not apply; we leave the theory for these two cases to be considered elsewhere.

\section{Appendix}

\begin{lemma}
The space $V$ with the inner product (\ref{eq:vnorm-1}) is
a real Hilbert space, $V$ is dense in $H$, and $V$ is separable.\end{lemma}
\begin{proof}
First, if $(z,z)_{V}=0$, where $z=[w,w_{0},w_{l}]$ , then $w(x)$
is a constant function and $w_{0}=w_{l}=0.$ The compatibility condition $ w(0) = w_0 $, $ w(l) = w_l $ (\ref{eqn:displacement_cc})
implies $w(x)=0$ for all $x$, and so $z=0$. It is straightforward
to show that $(\cdot,\cdot)_{V}$ satisfies the remaining properties of
an inner product.

Next, let $\left\{ z^{n}\right\} \subset V$ be a Cauchy sequence,
where $z^{n}=[w^{n},w_{0}^{n},w_{l}^{n}]$. Therefore, $[w_{x}^{n},w_{0}^{n},w_{l}^{n}]$
is a Cauchy sequence in $L^{2}(0,l)\times\mathbb{R}^{2},$ and so
there exists $[q,w_{0},w_{l}]\in L^{2}(0,l)\times\mathbb{R}^{2}$
such that
\[
w_{x}^{n}\rightarrow q\,\,\text{in}\,\,L^{2}(0,l),\,\,\,\,\,\,w_{0}^{n}\rightarrow w_{0},\,\,\,\,\,\,w_{l}^{n}\rightarrow w_{l}.
\]
Define $w$ by $w(x)=w_{0}+\int_{0}^{x}q(\eta)d\eta$. Then $w\in H^{1}(0,l),\,w_{x}=q,$
and $w(0)=w_{0}.$ Also, $w(l)=w_{l}$ since
\[
w(l)=\lim_{n\to\infty} w_{0}^{n}+\int_{0}^{l}w_{x}^{n}(\eta) \, d\eta=\lim_{n\to\infty}w_{l}^{n}=w_{l}.
\]
Therefore $z=[w,w_{0},w_{l}]$ satisfies the displacement compatibility
condition and $z^{n}$ converges in $V$ to $z\in V$. This shows
$V$ is a Hilbert space.

To show $V$ is dense in $H$, let $z=[w,w_{0},w_{l}]\in H$ and
define
\[
g(x)=w_{0}+l^{-1}(w_{l}-w_{0})x.
\]
Note that $g(0)=w_{0}$ and $g(l)=w_{l}$. Since $H_{0}^{1}(0,l)$
is dense in $L^{2}(0,l)$, there exists a sequence $q_{n}\in H_{0}^{1}(0,l)$
such that $q_{n}\rightarrow w-g$ in $L^{2}$. Define
\[
z_{n}=[q_{n}+g, w_{0}, w_{l}].
\]
Due to the properties of $q_{n}$ and $g$, we have $z_{n}\in V$
for all $n$ and also $z_{n}\rightarrow z$ in $H$ as $n\rightarrow\infty$.
This proves $V$ is dense in $H$.

To show $V$ is separable, let $z_{n}=[\sin(n\pi l^{-1}x),0,0]\in V$
and define $M\subset V$ by
\[
M=\{[x,0,l],[l-x,l,0]\}\cup\{z_{n}\}_{n=1}^{\infty}.
\]
Let $z=[w,w_{0},w_{l}]\in M^{\perp}$. We show $z=0$ so that the
span of $M$ is dense in $V$, and therefore $V$ is separable . First, the equations
\[
(z,[x,0,l])_{V}=0,  \quad  \left(z,[l-x,l,0]\right)_{V}=0
\]
imply
\[
\beta^{2}(w_{l}-w_{0})+k_{l}lw_{l}=0,  \quad  -\beta^{2}(w_{l}-w_{0})+k_{0}lw_{0}=0.
\]
It can be checked that the only solution of these questions is $w_{0}=w_{l}=0$,
and therefore $z=[w,0,0]$. The compatibility condition also gives
$w(0)=w(l)=0$, i.e., $w\in H_{0}^{1}(0,l)$. Finally, $\left\{ \sin(n\pi l^{-1}x)\right\} _{n=1}^{\infty}$ is
an orthogonal basis for $H_{0}^{1}(0,l)$, and therefore $(z,z_{n})_{V}=0$
for all $n$ implies $z=0.$ This proves $V$ is separable.\end{proof}


\bibliographystyle{siamplain}
\bibliography{citations_for_conference_paper}

\begin{thebibliography}{10}

\bibitem{AmsallemHetmaniuk14}
{\sc D.~Amsallem and U.~Hetmaniuk}, {\em Error estimates for {G}alerkin
  reduced-order models of the semi-discrete wave equation}, ESAIM Math. Model.
  Numer. Anal., 48 (2014), pp.~135--163,
  \url{https://doi.org/10.1051/m2an/2013099}.

\bibitem{Antoulas05}
{\sc A.~C. Antoulas}, {\em Approximation of {L}arge-{S}cale {D}ynamical
  {S}ystems}, SIAM, Philadelphia, PA, 2005.

\bibitem{banks2012functional}
{\sc H.~T. Banks}, {\em A Functional Analysis Framework for Modeling,
  Estimation and Control in Science and Engineering}, CRC Press, 2012.

\bibitem{batten2010reduced}
{\sc B.~A. Batten and K.~A. Evans}, {\em Reduced-order compensators via
  balancing and central control design for a structural control problem},
  Internat. J. Control, 83 (2010), pp.~563--574.

\bibitem{battenmodel}
{\sc B.~A. Batten, H.~Shoori, J.~R. Singler, and M.~H. Weerasinghe}, {\em Model
  reduction of a nonlinear cable-mass {PDE} system with dynamic boundary
  input}, in Proceedings of the International Symposium on Mathematical Theory
  of Networks and Systems, no.~327-334, 2016.

\bibitem{BennerSachsVolkwein14}
{\sc P.~Benner, E.~Sachs, and S.~Volkwein}, {\em Model order reduction for
  {PDE} constrained optimization}, in Trends in {PDE} constrained optimization,
  vol.~165 of Internat. Ser. Numer. Math., Birkh\"auser/Springer, Cham, 2014,
  pp.~303--326, \url{https://doi.org/10.1007/978-3-319-05083-6_19}.

\bibitem{bui2008parametric}
{\sc T.~Bui-Thanh, K.~Willcox, and O.~Ghattas}, {\em Parametric reduced-order
  models for probabilistic analysis of unsteady aerodynamic applications}, AIAA
  Journal, 46 (2008), pp.~2520--2529.

\bibitem{BurnsCliff14}
{\sc J.~A. Burns and E.~M. Cliff}, {\em Control of hyperbolic {PDE} systems
  with actuator dynamics}, in Proceedings of IEEE Conference on Decision and
  Control, 2014, pp.~2864--2869,
  \url{https://doi.org/10.1109/CDC.2014.7039829}.

\bibitem{burns1998reduced}
{\sc J.~A. Burns and B.~B. King}, {\em A reduced basis approach to the design
  of low-order feedback controllers for nonlinear continuous systems}, J. Vib.
  Control, 4 (1998), pp.~297--323.

\bibitem{BurnsZietsman12}
{\sc J.~A. Burns and L.~Zietsman}, {\em On the inclusion of actuator dynamics
  in boundary control of distributed parameter systems}, in IFAC Proceedings
  Volumes, 4th IFAC Workshop on Lagrangian and Hamiltonian Methods for Non
  Linear Control, vol.~45, 2012, pp.~138 -- 142,
  \url{http://dx.doi.org/10.3182/20120829-3-IT-4022.00039}.

\bibitem{BurnsZietsman16}
{\sc J.~A. Burns and L.~Zietsman}, {\em Control of a thermal fluid heat
  exchanger with actuator dynamics}, in Proceedings of IEEE Conference on
  Decision and Control, 2016, pp.~3131--3136,
  \url{https://doi.org/10.1109/CDC.2016.7798738}.

\bibitem{conrad2001uniform}
{\sc F.~Conrad and A.~Mifdal}, {\em Uniform stabilization of a hybrid system
  with a class of nonlinear feedback laws}, Adv. Math. Sci. Appl., 11 (2001),
  pp.~549--569.

\bibitem{curtain2001compactness}
{\sc R.~F. Curtain and A.~J. Sasane}, {\em Compactness and nuclearity of the
  {H}ankel operator and internal stability of infinite-dimensional state linear
  systems}, Internat. J. Control, 74 (2001), pp.~1260--1270.

\bibitem{daescu2007efficiency}
{\sc D.~N. Daescu and I.~M. Navon}, {\em Efficiency of a {POD}-based reduced
  second-order adjoint model in 4{D}-var data assimilation}, Internat. J.
  Numer. Methods Fluids, 53 (2007), pp.~985--1004.

\bibitem{fang2014reduced}
{\sc F.~Fang, T.~Zhang, D.~Pavlidis, C.~Pain, A.~Buchan, and I.~Navon}, {\em
  Reduced order modelling of an unstructured mesh air pollution model and
  application in 2d/3d urban street canyons}, Atmospheric Environment, 96
  (2014), pp.~96--106.

\bibitem{fourrier2013regularity}
{\sc N.~Fourrier and I.~Lasiecka}, {\em Regularity and stability of a wave
  equation with a strong damping and dynamic boundary conditions}, Evol. Equ.
  Control Theory, 2 (2013), pp.~631--667.

\bibitem{glover1988realisation}
{\sc K.~Glover, R.~F. Curtain, and J.~R. Partington}, {\em Realisation and
  approximation of linear infinite-dimensional systems with error bounds}, SIAM
  J. Control Optim., 26 (1988), pp.~863--898.

\bibitem{GongWangWang17}
{\sc Y.~Gong, Q.~Wang, and Z.~Wang}, {\em Structure-preserving {G}alerkin {POD}
  reduced-order modeling of {H}amiltonian systems}, Computer Methods in Applied
  Mechanics and Engineering, 315 (2017), pp.~780--798.

\bibitem{guiver2014model}
{\sc C.~Guiver and M.~R. Opmeer}, {\em Model reduction by balanced truncation
  for systems with nuclear {H}ankel operators}, SIAM J. Control Optim., 52
  (2014), pp.~1366--1401.

\bibitem{gunzburger2016ensemble}
{\sc M.~Gunzburger, N.~Jiang, and M.~Schneier}, {\em An {E}nsemble-{P}roper
  {O}rthogonal {D}ecomposition {M}ethod for the {N}onstationary
  {N}avier--{S}tokes {E}quations}, SIAM J. Numer. Anal., 55 (2017),
  pp.~286--304, \url{https://doi.org/10.1137/16M1056444}.

\bibitem{gunzburger2005reduced}
{\sc M.~Gunzburger and H.-C. Lee}, {\em Reduced-order modeling of
  {N}avier-{S}tokes equations via centroidal {V}oronoi tessellation}, in Recent
  advances in adaptive computation, vol.~383 of Contemp. Math., Amer. Math.
  Soc., Providence, RI, 2005, pp.~213--224,
  \url{https://doi.org/10.1090/conm/383/07166}.

\bibitem{herkt2013convergence}
{\sc S.~Herkt, M.~Hinze, and R.~Pinnau}, {\em Convergence analysis of
  {G}alerkin {POD} for linear second order evolution equations}, Electron.
  Trans. Numer. Anal., 40 (2013), pp.~321--337.

\bibitem{HuynhKnezevicPatera11}
{\sc D.~B.~P. Huynh, D.~J. Knezevic, and A.~T. Patera}, {\em A {L}aplace
  transform certified reduced basis method; application to the heat equation
  and wave equation}, C. R. Math. Acad. Sci. Paris, 349 (2011), pp.~401--405,
  \url{http://dx.doi.org/10.1016/j.crma.2011.02.003}.

\bibitem{ilak2010model}
{\sc M.~Ilak, S.~Bagheri, L.~Brandt, C.~W. Rowley, and D.~S. Henningson}, {\em
  Model reduction of the nonlinear complex {G}inzburg-{L}andau equation}, SIAM
  J. Appl. Dyn. Syst., 9 (2010), pp.~1284--1302.

\bibitem{King94}
{\sc B.~B. King}, {\em Modeling and control of a multiple component structure},
  J. Math. Systems Estim. Control, 4 (1994), p.~36.

\bibitem{komornik1994exact}
{\sc V.~Komornik}, {\em Exact controllability and stabilization}, RAM: Research
  in Applied Mathematics, Masson, Paris; John Wiley \& Sons, Ltd., Chichester,
  1994.
\newblock The multiplier method.

\bibitem{moore1981principal}
{\sc B.~C. Moore}, {\em Principal component analysis in linear systems:
  controllability, observability, and model reduction}, IEEE Trans. Automat.
  Control, 26 (1981), pp.~17--32.

\bibitem{morgul1998stabilization}
{\sc {\"O}.~Morg{\"u}l}, {\em Stabilization and disturbance rejection for the
  wave equation}, IEEE Trans. Automat. Control, 43 (1998), pp.~89--95.

\bibitem{Morris02}
{\sc K.~Morris}, {\em ${H}_\infty$-control of acoustic noise in a duct with a
  feedforward configuration}, in Proceedings of 15th International Symposium on
  Mathematical Theory of Networks and Systems, 2002, pp.~12--16.

\bibitem{Opmeer08}
{\sc M.~R. Opmeer}, {\em Nuclearity of {H}ankel operators for
  ultradifferentiable control systems}, Systems Control Lett., 57 (2008),
  pp.~913--918, \url{http://dx.doi.org/10.1016/j.sysconle.2008.04.007}.

\bibitem{pazy2012semigroups}
{\sc A.~Pazy}, {\em Semigroups of linear operators and applications to partial
  differential equations}, vol.~44, Springer Science \& Business Media, 2012.

\bibitem{PengMohseni16}
{\sc L.~Peng and K.~Mohseni}, {\em Symplectic model reduction of {H}amiltonian
  systems}, SIAM J. Sci. Comput., 38 (2016), pp.~A1--A27,
  \url{http://dx.doi.org/10.1137/140978922}.

\bibitem{pereyra2016model}
{\sc V.~Pereyra}, {\em Model order reduction with oblique projections for large
  scale wave propagation}, J. Comput. Appl. Math., 295 (2016), pp.~103--114.

\bibitem{scherpen1993balancing}
{\sc J.~M.~A. Scherpen}, {\em Balancing for nonlinear systems}, Systems Control
  Lett., 21 (1993), pp.~143--153.

\bibitem{Shoori14}
{\sc H.~E. Shoori~J.}, {\em An approach to reduced-order modeling and feedback
  control for wave energy converters}, master's thesis, Oregon State
  University, 2014,
  \url{http://ir.library.oregonstate.edu/xmlui/handle/1957/50823}.

\bibitem{varshney2009feedback}
{\sc A.~Varshney, S.~Pitchaiah, and A.~Armaou}, {\em Feedback control of
  dissipative {PDE} systems using adaptive model reduction}, AIChE journal, 55
  (2009), pp.~906--918.

\bibitem{zhang2016stabilization}
{\sc Z.~Zhang}, {\em Stabilization of the wave equation with variable
  coefficients and a dynamical boundary control}, Electron. J. Differential
  Equations,  (2016), pp.~Paper No. 27, 10.

\bibitem{ZhouDoyleGlover96}
{\sc K.~Zhou, J.~C. Doyle, and K.~Glover}, {\em Robust and {O}ptimal
  {C}ontrol}, Prentice-Hall, 1996.

\end{thebibliography}

\end{document}